\documentclass[a4paper]{article}

\usepackage{amsthm}
\usepackage{lmodern}

\usepackage{mathrsfs}
\usepackage{amsfonts}
\usepackage{enumerate}
\usepackage{enumitem}
\usepackage{ulem}
\usepackage{amssymb}
\usepackage{graphicx}
\usepackage{booktabs}
\usepackage{xcolor}
\usepackage{stmaryrd}
\usepackage{marvosym}
\usepackage{comment}

\newcommand{\doublehat}[1]{\widehat{#1}}

 \usepackage{amsmath}
 
 \if@mathematic
   \def\vec#1{\ensuremath{\mathchoice
                     {\mbox{\boldmath$\displaystyle\mathbf{#1}$}}
                     {\mbox{\boldmath$\textstyle\mathbf{#1}$}}
                     {\mbox{\boldmath$\scriptstyle\mathbf{#1}$}}
                     {\mbox{\boldmath$\scriptscriptstyle\mathbf{#1}$}}}}
\else
   \def\vec#1{\ensuremath{\mathchoice
                     {\mbox{\boldmath$\displaystyle#1$}}
                     {\mbox{\boldmath$\textstyle#1$}}
                     {\mbox{\boldmath$\scriptstyle#1$}}
                     {\mbox{\boldmath$\scriptscriptstyle#1$}}}}
\fi
    
\usepackage{geometry}
\theoremstyle{plain}
\newtheorem{thm}{Theorem}
\newtheorem{cor}{Corollary}
\newtheorem{lem}{Lemma}

\theoremstyle{remark}
\newtheorem{rem}{Remark}

\renewcommand{\matrix}[1]{\textbf{#1}}

\newcommand{\pderivative}[2]{\frac{\partial #1}{\partial #2}}
\newcommand{\jump}[1]{\left\llbracket #1 \right\rrbracket}
\newcommand{\average}[1]{\left\{\!\!\left\{#1\right\}\!\!\right\}}

\usepackage{hyperref}
\hypersetup{
    colorlinks, linkcolor={black},
    citecolor={black}, urlcolor={black}
}

\numberwithin{equation}{section}

\begin{document}

\title{An Entropy Stable Finite Volume Scheme for the Equations of Shallow Water Magnetohydrodynamics}
\author{Andrew R.~Winters\footnote{Mathematical Institute, University of Cologne, 50931, Cologne, Germany, awinters@math.uni-koeln.de} \and Gregor J.~Gassner\footnote{Mathematical Institute, University of Cologne, 50931, Cologne, Germany}}
\date{}


\maketitle

\begin{abstract}
In this work, we design an entropy stable, finite volume approximation for the shallow water magnetohydrodynamics (SWMHD) equations. The method is novel as we design an affordable analytical expression of the numerical interface flux function that exactly preserves the entropy, which is also the total energy for the SWMHD equations. To guarantee the discrete conservation of entropy requires a special treatment of a consistent source term for the SWMHD equations. With the goal of solving problems that may develop shocks, we determine a dissipation term to guarantee entropy stability for the numerical scheme. Numerical tests are performed to demonstrate the theoretical findings of entropy conservation and robustness.
\end{abstract}


\section{Introduction}\label{intro}

Systems of conservation laws that model a physical system or process often possess an additional conserved quantity, the entropy, which is conserved for smooth solutions but increases (or decreases according to the sign convention adopted) in the presence of shocks. A method is said to be entropy conservative if the local changes of entropy are the same as predicted by the entropy conservation law. The approximation is said to be entropy stable if it produces more entropy than an entropy conservative scheme.

We introduce the important distinction between entropy conservation and stability because there is a problem with entropy conservative formulations. They may suffer breakdown if used without dissipation to capture shocks. Physically, entropy must be dissipated at a shock. However, an isentropic algorithm does not allow the capture of this physical process, which generates large amplitude oscillations around the shock \cite{carpenter_esdg}. A more dire issue is that entropy conservative formulations can not converge to the weak solution as there is no mechanism to admit the dissipation physically required at the shock.

Tadmor \cite{tadmor1984,tadmor1987} was the first to introduce the idea of entropy conservation to design stable numerical approximations of nonlinear hyperbolic conservation laws. At a semi-discrete level the principle is that the discrete flux function satisfies discrete conservation laws of the conservative variables as well as the scalar conservation law for entropy in a discrete sense. Tadmor's flux function performs well for smooth data but can be made stable for shock problems \cite{tadmor2003}. But, Tadmor's flux function, which involves an integral in phase space, is numerically expensive. For large scale simulations we want a more practical, computationally tractable, entropy conserving flux.

Thus, in this paper, it is our goal to develop affordable, entropy stable methods for the shallow water magnetohydrodynamic (SWMHD) model. There is recent work for optimal, entropy stable approximations for the Euler equations from Ismail and Roe \cite{ismail2009} which was extended to arbitrary order with a discontinuous Galerkin (DG) spectral element formulation by Carpenter et al. \cite{carpenter_esdg}. Also, there is recent work on entropy-stable, high-order, and well-balanced approximations for the shallow water equations \cite{gassner_skew_burgers,gassner2014,winters2014}.

The remainder of this paper is organized as follows: Sec. \ref{GoverningEqns} provides an introduction to the SWMHD equations. In Sec. \ref{FVDiscretization} we briefly describe the finite volume discretization used. Sec. \ref{EntropyAnalysis} defines the necessary variables and analytical tools to discuss the entropy in a mathematically rigorous way. We derive an entropy conserving numerical flux in Sec. \ref{Sec:ConservingFlux} and discuss the stabilized flux in Sec. \ref{Sec:StableFlux}. Numerical results are presented in Sec. \ref{NumericalResults}, where we verify theoretical predictions. Sec. \ref{conclusion} presents concluding remarks. {\color{black}{Finally, in the Appendix we provide entropy stable fluxes that can be used for two dimensional computations.}}

\section{The Shallow Water Magnetohydrodynamic Model}\label{GoverningEqns}

The shallow water magnetohydrodynamics (SWMHD) equations were first proposed by Gilman as a model for phenomena in the solar tachocline \cite{gilman2000}, which is the thin layer in a star between the outer turbulent convection zone, and the quiescent interior where heat transfer is predominantly radiative. {\color{black}{The tachocline has many interesting physical applications such as the formation of sunspots  \cite{desterck2001}. Tachoclines can also serve as a useful model of solar dynamo action \cite{dikpati2001a,dikpati2001,gilman2000}.}}

{\color{black}{As was first shown by Gilman \cite{gilman2000} with complete details presented by Rossmanith \cite{rossmanith2002} the SWMHD equations are derived from the ideal MHD equations by integrating in the $z-$direction under the assumption of incompressibility, a two dimensional variation of the flow variables, and magnetohydrostatic equilibrium in the $z-$direction:
\begin{equation}
\begin{aligned}
\rho &\equiv constant,\\
\pderivative{}{z}\left(p + \frac{\|\vec{B}\|^2}{2}\right) &= \rho g,\\
p + \frac{\|\vec{B}\|^2}{2} &= \rho g z,\\
\int_0^h \left(p + \frac{\|\vec{B}\|^2}{2}\right)\,dz &= \rho g \frac{h^2}{2}.
\end{aligned}
\end{equation}
Under these assumptions and the divergence-free constraint $\nabla\cdot(h\vec{B})$ the SWMHD equations can be cast into the form of a system of conservation laws \cite{desterck2001,rossmanith2002}.}}
The one dimensional SWMHD equations comprise a hyperbolic system written in the conservative form
\begin{equation}\label{SWMHD1D}
\begin{aligned}
\pderivative{\vec{u}}{t} + \pderivative{\vec{f}}{x} = 
\begin{bmatrix}
h\\
hv_1\\
hv_2\\
hB_1\\
hB_2\\
\end{bmatrix}_t + \begin{bmatrix}
hv_1 \\
hv_1^2 +\frac{1}{2}gh^2 - hB_1^2 \\
hv_1v_2-hB_1B_2 \\
0 \\
hv_1B_2 - hv_2B_1\\
\end{bmatrix}_x
&=
\begin{bmatrix}
0 \\
0 \\
0 \\
0 \\
0 \\
\end{bmatrix},\\\quad \pderivative{(hB_1)}{x} &= 0,
\end{aligned}
\end{equation}
where the final equation is the one dimensional involution condition on the magnetic field $\vec{B}$. The variables $v_1$, $v_2$, and $\vec{B}$ in \eqref{SWMHD1D} are the horizontal components of the fluid velocity and magnetic field, $h$ is the fluid layer depth, and $g$ is the constant gravitational acceleration.

Hyperbolic systems, like the SWMHD equations, are typically solved numerically with Godunov-type methods \cite{leveque2002}. These methods require the solution of a Riemann problem at element interfaces. Unfortunately, the one-dimensional SWMHD equations are degenerate \cite{desterck2001}, just like the ideal MHD equations \cite{powell1994,roe1996}, due to the involution, $\nabla\cdot(h\vec{B}) = 0$. For example, if we consider the one dimensional problem we see from \eqref{SWMHD1D} that $\partial_t(hB_1) = 0$. Thus, there is no propagating Riemann wave associated with the conservative variable $hB_1$. To alleviate such degeneracies we are interested in relaxing the involution $\nabla\cdot(h\vec{B}) = 0$ such that a conventional Riemann solver based method may be used for the approximation of the SWMHD system. 

To accomplish this we weaken the divergence condition to require that $\nabla\cdot(hB)\approx 0,$ i.e., the divergence condition is to be maintained up to finite precision error.
Dellar \cite{dellar2002} showed that for the relaxed assumption of the divergence condition we obtain a forced version of the SWMHD model 
\begin{equation}\label{SWMHD1DForced}
\begin{aligned}
\pderivative{\vec{u}}{t} + \pderivative{\vec{f}}{x} = \begin{bmatrix}
h\\
hv_1\\
hv_2\\
hB_1\\
hB_2\\
\end{bmatrix}_t + \begin{bmatrix}
hv_1 \\
hv_1^2 +\frac{1}{2}gh^2 - hB_1^2 \\
hv_1v_2-hB_1B_2 \\
0 \\
hv_1B_2 - hv_2B_1\\
\end{bmatrix}_x
&=
- \pderivative{(hB_1)}{x}\begin{bmatrix}
0 \\
0 \\
0 \\
v_1 \\
v_2 \\
\end{bmatrix} = \vec{s},
\end{aligned}
\end{equation}
where the source term is proportional to the divergence condition. Importantly, the source term in \eqref{SWMHD1DForced} alters the system to make $hB_1$ an advected scalar while maintaining the conservation of mass and momentum. The forced equations \eqref{SWMHD1DForced} can be derived from a Hamiltonian approach \cite{dellar2002}, so the system \eqref{SWMHD1DForced} conserves the total energy. The conservation of total energy is important when we develop the entropy conserving flux in Sec. \ref{Sec:ConservingFlux}. Also, the source term serves to restore Galilean invariance when $\nabla\cdot(h\vec{B})\ne 0$ \cite{dellar2002}. Finally, the source term in \eqref{SWMHD1DForced} is analogous to the Janhunen source term for the ideal MHD equations \cite{janhunen2000}.

It is important to note that the source term in \eqref{SWMHD1DForced} preserves the conservation of the momentum and energy, while enforcing the involution condition on $h\vec{B}$. Preserving the conservative properties of the fluid equations removes a significant drawback of the other common proposed Powell type source term for the SWMHD equations. The Powell source term renders the equations of momentum and energy for SWMHD non-conservative \cite{dellar2002}. However, according to the Lax-Wendroff theorem \cite{lax1960} only conservative schemes can be expected to compute the correct jump conditions and propagation speed for a discontinuous solution. Difficulties of the non-conservative formulation obtaining the correct weak solution has been documented in the literature for the ideal MHD equations \cite{toth2000}. 

\section{Finite Volume Discretization}\label{FVDiscretization}

The finite volume (FV) method is a discretization technique for partial differential equations especially useful for the approximation of systems of conservations laws. The finite volume method is designed to approximate conservation laws in their original form, e.g.,
\begin{equation}
\int_V\vec{u}_t\,dx + \int_{\partial V} \vec{f}\cdot\hat{\vec{n}}\,dS = 0.
\end{equation}
For instance, in one spatial dimension we break the interval into non-overlapping intervals (the ``volumes'')
\begin{equation}
V_i = \left[x_{i-\tfrac{1}{2}},x_{i+\tfrac{1}{2}}\right],
\end{equation}
and the integral equation of a balance law, with a source term, becomes
\begin{equation}\label{FVSource}
\frac{d}{dt}\int_{x_{i-{1}/{2}}}^{x_{i+{1}/{2}}} \vec{u}\,dx + \vec{f}^*\left(x_{i+{1}/{2}}\right) - \vec{f}^*\left(x_{i-{1}/{2}}\right) = \int_{x_{i-{1}/{2}}}^{x_{i+{1}/{2}}} \vec{s}\,dx
\end{equation}
At this point, we make approximations. A common approximation is to assume that the solution and the source term are constant within the volume. Then we determine, for example, what is analogous to a midpoint quadrature approximation of the solution integral
\begin{equation}
\int_{x_{i-{1}/{2}}}^{x_{i+{1}/{2}}} \vec{u}\,dx \approx \int_{x_{i-{1}/{2}}}^{x_{i+{1}/{2}}} \vec{u}_i\,dx = \vec{u}_i\Delta x_i.
\end{equation}
Note that the solution is typically discontinuous at the boundaries of the volumes. To resolve this, we introduce the idea of a ``numerical flux,'' $\vec{f}^*(\vec{u}^L,\vec{u}^R)$, often derived from the (approximate) solution of a Riemann problem. That is, $\vec{f}^*$ is a function that takes the two states of the solution at an element interface and returns a single flux value. For consistency, we require that 
\begin{equation}\label{consistency}
\vec{f}^*(\vec{u},\vec{u}) = \vec{f},
\end{equation}
that is, the numerical flux is equivalent to the physical flux if the states on each side of the interface are identical. A significant portion of this paper is devoted to the derivation of a numerical flux that conserves the discrete entropy of the system for the shallow water MHD equations. So we defer the discussion of the numerical flux to Sec. \ref{EntropyFlux}. 

We must also address how to discretize the source term $\vec{s}$ in \eqref{FVSource}. There is a significant amount of freedom in the source term discretization. But, previous work by Fjordholm et al. \cite{fjordholm2011} demonstrated that designing entropy stable methods for the shallow water equations with discontinuous bottom topography required a special treatment of the source term. In the later derivations a consistent source term discretization necessary for entropy conservation will reveal itself. Thus, we also defer the discussion of the discrete treatment of the source term to Sec. \ref{Sec:ConservingFlux}.

\section{Entropy Analysis}\label{EntropyAnalysis}

In this section we define the entropy variables and entropy Jacobian necessary to develop an entropy stable approximation for the SWMHD equations. We note that one can find a fully general and detailed description of entropy stability theory in, for example, \cite{barth99,fjordholm2012}. 

For the SWMHD equations the total energy acts as an entropy function, \cite{dellar2002},
\begin{equation}\label{entropyFunction}
U\left(\vec{u}\right) = \frac{1}{2}\left(gh^2+hv_1^2+hv_2^2+hB_1^2+hB_2^2\right),
\end{equation}
where $\vec{u}$ is the vector of conserved variables. From the total entropy \eqref{entropyFunction} we define the set of entropy variables
\begin{equation}\label{entropyVariables}
U_{\vec{u}} := \vec{q} = \left(gh-\frac{1}{2}\left\{v_1^2+v_2^2+B_1^2+B_2^2\right\}\,,\,v_1\,,\,v_2\,,\,B_1\,,\,B_2\right)^T.
\end{equation}
The entropy variables \eqref{entropyVariables} are equipped with the symmetric positive definite (s.p.d) Jacobian matrices 
\begin{equation}\label{entropyJacobianInverse}
\matrix{H}^{-1} =\vec{q}_{\vec{u}} = \frac{1}{h}\begin{bmatrix} c^2+v_1^2+v_2^2+B_1^2+B_2^2 & -v_1 &  -v_2 & -B_1 & -B_2 \\[0.1cm]
-v_1 & 1 & 0 & 0 & 0 \\[0.1cm]
-v_2 & 0 & 1 & 0 & 0 \\[0.1cm]
-B_1 & 0 & 0 & 1 & 0 \\[0.1cm]
-B_2 & 0 & 0 & 0 & 1
\end{bmatrix},
\end{equation}
and
\begin{equation}\label{entropyJacobian}
\matrix{H} =\vec{u}_{\vec{q}} = \frac{1}{g}
\begin{bmatrix} 
1 & v_1 & v_2 & B_1 & B_2 \\[0.1cm]
v_1 & v_1^2 + c^2 & v_1 v_2 & v_1 B_1 & v_1 B_2 \\[0.1cm]
v_2 & v_1 v_2 & v_2^2 + c^2 & v_2 B_1 & v_2 B_2 \\[0.1cm]
B_1 & v_1 B_1 & v_2 B_1 & B_1^2 + c^2 & B_1 B_2 \\[0.1cm]
B_2 & v_1 B_2 & v_2 B_2 & B_1 B_2 & B_2^2 + c^2
\end{bmatrix},
\end{equation}
where $c^2 = gh$ is the standard wave speed for the shallow water equations. The fact that the matrices $\matrix{H}$ and $\matrix{H}^{-1}$ are s.p.d. is a direct consequence that the entropy function \eqref{entropyFunction} is a strongly convex, injective mapping \cite{dellar2002}. For the derivations that follow we also require the flux Jacobian in the $x-$direction:
\begin{equation}\label{fluxJacobian}
\matrix{A} =\vec{f}_{\vec{u}} = \begin{bmatrix} 0 & 1 & 0 & 0 & 0  \\[0.1cm]
c^2 - v_1^2 + B_1^2 & 2v_1 & 0 & -2B_1 & 0 \\[0.1cm]
-v_1v_2+B_1B_2 & v_2 & v_1 & -B_2 & -B_1 \\[0.1cm]
0 & 0 & 0 & 0 & 0 \\[0.1cm]
v_2B_1 - v_1B_2 & B_2 & -B_1 & -v_2 & v_1
\end{bmatrix}.
\end{equation}
Because it will be of use in later derivations we also compute the entropy potential 
\begin{equation}\label{entropyPotential}
\psi = \vec{q}\cdot\vec{f} - F =  \frac{1}{2}gh^2v_1 - hB_1\left(v_1B_1 + v_2B_2\right),
\end{equation}
where $F$ is the entropy flux derived from the identity $U_{\vec{u}}\matrix{A} = F_{\vec{u}}$, \cite{fjordholm2011},
\begin{equation} 
F = gh^2v_1 + \frac{1}{2}\left(hv_1^3 + hv_1v_2^2 + hv_1B_2^2 - hv_1B_1^2\right) - hv_2B_1B_2.
\end{equation}
\subsection{Discrete Entropy Conservation for the 1D SWMHD Equations}\label{Sec:DiscreteEntropy}
We introduce the concept of discrete entropy conservation. Let's assume that we have two adjacent states $(L,R)$ with cell areas $(\Delta x_L,\Delta x_R)$. We discretize the SWMHD equations semi-discretely and examine the approximation at the $i+\tfrac{1}{2}$ interface. We suppress the interface index unless it is necessary for clarity.
\begin{equation}\label{FVupdate}
\begin{aligned}
\Delta x_L\pderivative{\vec{u}_L}{t} &= \vec{f}_L - \vec{f}^*  + \Delta x_L{\vec{s}}_{i+\tfrac{1}{2}},\\
\Delta x_R\pderivative{\vec{u}_R}{t} &= \vec{f}^*-\vec{f}_R + \Delta x_R{\vec{s}}_{i+\tfrac{1}{2}}.
\end{aligned}
\end{equation}
We interpret the update \eqref{FVupdate} as a finite volume scheme where we have left and right cell-averaged values separated by a common flux interface.

We premultiply the expressions \eqref{FVupdate} by the entropy variables to convert to entropy space. From the chain rule we know that $U_t = \vec{q}^T\vec{u}_t$, hence a semi-discrete entropy update is
\begin{equation}\label{EntropyUpdate}
\begin{aligned}
\Delta x_L\pderivative{U_L}{t} &= \vec{q}_L^T\left(\vec{f}_L - \vec{f}^*  + \Delta x_L{\vec{s}}_{i+\tfrac{1}{2}}\right), \\
\Delta x_R\pderivative{U_R}{t} &= \vec{q}_R^T\left(\vec{f}^*-\vec{f}_R + \Delta x_R{\vec{s}}_{i+\tfrac{1}{2}}\right).
\end{aligned}
\end{equation}
If we denote the jump in a state as $\jump{\cdot} = (\cdot)_R - (\cdot)_L$ and the average of a state as $\average{\cdot} = ((\cdot)_R + (\cdot)_L)/2$, then we find the total element update to the entropy is
\begin{equation}\label{TotalUpdate}
\pderivative{}{t}(\Delta x_L U_L + \Delta x_R U_R) = -\jump{\vec{q}\cdot\vec{f}} + \jump{\vec{q}}^T\vec{f}^*  + 2\average{\Delta x\vec{q}}{\vec{s}}_{i+\tfrac{1}{2}}.
\end{equation}
We want the discrete entropy update to satisfy the discrete entropy conservation law. To achieve this, we require
\begin{equation}\label{entropyConservationCondition1}
2\average{\Delta x\vec{q}}^T{\vec{s}}_{i+\tfrac{1}{2}}-\jump{\vec{q}\cdot\vec{f}} + \jump{\vec{q}}^T\vec{f}^* = -\jump{F}.
\end{equation}
We combine the known entropy potential $\psi$ in \eqref{entropyPotential} and the linearity of the jump operator to rewrite the entropy conservation condition \eqref{entropyConservationCondition1} as 
\begin{equation}\label{entropyConservationCondition2}
\jump{\vec{q}}^T\vec{f}^* =  \frac{1}{2}\jump{gh^2v_1} - \jump{hB_1\left(v_1B_1 + v_2B_2\right)} - 2\average{\Delta x\vec{q}}^T{\vec{s}}_{i+\tfrac{1}{2}}.
\end{equation}
We denote the constraint \eqref{entropyConservationCondition2} as the discrete entropy conserving condition. However, this is a single condition on the vector $\vec{f}^*$, so there are many potential solutions for the entropy conserving flux. However, we have the additional requirement that the numerical flux must be consistent \eqref{consistency}. We develop an affordable entropy conserving flux function in Sec. \ref{EntropyFlux}. We also defer the discretization of the source term ${\vec{s}_{i+1/2}}$ to to the next section because special care must be taken to guarantee that \eqref{entropyConservationCondition2} is satisfied.

\section{Derivation of an Entropy Stable Numerical Flux}\label{EntropyFlux}

With the necessary entropy variable and Jacobian definitions as well as the formulation of the discrete entropy conserving condition \eqref{entropyConservationCondition2} we are ready to derive an affordable entropy conserving numerical flux in Sec. \ref{Sec:ConservingFlux}. As we previously noted, entropy conserving methods may suffer breakdown in the presence of shocks \cite{carpenter_esdg}. Thus, in Sec. \ref{Sec:StableFlux}, we will design an entropy stable Riemann solver that uses the entropy conserving flux as a base and incorporates a dissipation term required for stability.

\subsection{Entropy Conserving Numerical Flux for 1D SWMHD}\label{Sec:ConservingFlux}

In this section we define an affordable, entropy conserving numerical flux for the one dimensional SWMHD equations.

\begin{thm}(Entropy Conserving Numerical Flux)
If we discretize the source term in the finite volume method to contribute to each element as 
\begin{equation}
\vec{s}_i = \frac{1}{2}\left(\vec{s}_{i+\tfrac{1}{2}} + \vec{s}_{i-\tfrac{1}{2}}\right) = -\frac{1}{2}\left(
\jump{hB_1}_{i+\tfrac{1}{2}}
\begin{bmatrix}
0\\
0\\
0\\
\frac{\average{v_1B_1}}{\average{\Delta x B_1}}\\[0.1cm]
\frac{\average{v_2B_2}}{\average{\Delta x B_2}}
\end{bmatrix}_{i+\tfrac{1}{2}}
+
\jump{hB_1}_{i-\tfrac{1}{2}}
\begin{bmatrix}
0\\
0\\
0\\
\frac{\average{v_1B_1}}{\average{\Delta x B_1}}\\[0.1cm]
\frac{\average{v_2B_2}}{\average{\Delta x B_2}}
\end{bmatrix}_{i-\tfrac{1}{2}}
\right),
\end{equation}
then we can determine a discrete, entropy conservative flux to be
\begin{equation}
\vec{f}^{*,ec} = \begin{bmatrix}
\average{h}\average{v_1} \\[0.1cm]
\average{h}\average{v_1}^2 + \frac{1}{2}g\average{h^2} - \average{hB_1}\average{B_1}\\[0.1cm]
\average{h}\average{v_1}\average{v_2} - \average{hB_1}\average{B_2} \\[0.1cm]
\average{h}\average{v_1}\average{B_1}-\average{hB_1}\average{v_1} \\[0.1cm]
\average{h}\average{v_1}\average{B_2}-\average{hB_1}\average{v_2}
\end{bmatrix}.
\end{equation}
\end{thm}

\begin{proof}
To derive an affordable entropy conservative flux for the one dimensional SWMHD equations we first expand the discrete entropy conserving condition \eqref{entropyConservationCondition2} componentwise to find
\begin{equation}\label{entropyCondition}
\begin{aligned}
\jump{gh -\frac{1}{2}\left\{v_1^2+v_2^2+B_1^2+B_2^2\right\}}f_1^* &+ \jump{v_1}f_2^* + \jump{v_2}f_3^* + \jump{B_1}f_4^* + \jump{B_2}f_5^* = \\ & \frac{1}{2}\jump{gh^2v_1} - \jump{hB_1\left(v_1B_1 + v_2B_2\right)} - 2\average{\Delta x\vec{q}}^T{\vec{s}}_{i+\tfrac{1}{2}},
\end{aligned}
\end{equation}
To determine the unknown flux components of $\vec{f}^*$ we use \eqref{entropyCondition} to obtain a system of equations for each of the linear jump terms. For this expansion step we make repeated use of the linear jump operator properties:
\begin{equation}
\jump{ab}  = \average{a}\jump{b} + \average{b}\jump{a},\qquad \jump{a^2} = 2\average{a}\jump{a}.
\end{equation}
We expand the terms of \eqref{entropyCondition} to linear jump terms of the primitive variables $h$, $v_1$, $v_2$, $B_1$, and $B_2$. This is straightforward in all but the last two terms on the right hand side
\begin{equation}\label{eq:entropyConstraint2}
\begin{aligned}
&g\jump{h}f_1^*+\left(f_2^*-\average{v_1}f_1^*\right)\jump{v_1}+\left(f_3^*-\average{v_2}f_1^*\right)\jump{v_2} + \left(f_4^*-\average{B_1}f_1^*\right)\jump{B_1}+\left(f_5^*-\average{B_2}f_1^*\right)\jump{B_2}\\
&= \frac{1}{2}g\average{h^2}\jump{v_1} + g\average{v_1}\average{h}\jump{h} - \jump{hB_1(v_1B_1+v_2B_2)} - 2\average{\Delta x\vec{q}}^T{\vec{s}}_{i+\tfrac{1}{2}}.
\end{aligned}
\end{equation}

Care must be taken when we expand the the second term on the right hand side of the entropy condition \eqref{eq:entropyConstraint2}. If we naively split the jump terms all the way to individual linear jump terms it will introduce averages of the magnetic field variables $\vec{B}$ on the jump in fluid height $h$. However, the continuity equation in the SWMHD model contains no influence from the magnetic field. Therefore, it would be impossible for the resulting numerical flux to be consistent with the physical flux. 

To remove this inconsistency we split the magnetic field jumps on the right hand side of \eqref{eq:entropyConstraint2} in such a way that cancellation with the source term is possible. Thus, rather than fully expanding each jump term down to the primitive variable level, it is useful to leave the product $hB_1$ term as a jump. This choice is motivated by the structure of the source term because 
\begin{equation}
    \pderivative{}{x}\!\left(hB_1\right) \approx \frac{\jump{hB_1}}{\Delta x}. 
\end{equation}
We expand the second term on the right hand side of \eqref{eq:entropyConstraint2} as follows:
\begin{equation}\label{RHSTerm2}
\begin{aligned}
\jump{hB_1(v_1B_1+v_2B_2)} &= \jump{hv_1B_1^2} + \jump{hv_2B_1B_2}, \\[0.1cm]
&= \average{hB_1}\jump{v_1B_1} + \average{v_1B_1}\jump{hB_1} + \average{hB_1}\jump{v_2B_2} + \average{v_2B_2}\jump{hB_1},\\[0.1cm]
&= \big(\average{v_1B_1} + \average{v_2B_2}\big)\jump{hB_1} + \average{hB_1}\jump{v_1B_1} + \average{hB_1}\jump{v_2B_2},\\[0.1cm]
&= \big(\average{v_1B_1} + \average{v_2B_2}\big)\jump{hB_1} + \average{hB_1}\big(\average{v_1}\jump{B_1} + \average{B_1}\jump{v_1}\big) \\
&\qquad\qquad+ \average{hB_1}\big(\average{v_2}\jump{B_2} + \average{B_2}\jump{v_2}\big).\\[0.1cm]
\end{aligned}
\end{equation}
We substitute \eqref{RHSTerm2} into the expanded entropy constraint \eqref{eq:entropyConstraint2} to obtain
\begin{equation}\label{eq:entropyConstraint3}
\begin{aligned}
&g\jump{h}f_1^*+\left(f_2^*-\average{v_1}f_1^*\right)\jump{v_1}+\left(f_3^*-\average{v_2}f_1^*\right)\jump{v_2} + \left(f_4^*-\average{B_1}f_1^*\right)\jump{B_1}+\left(f_5^*-\average{B_2}f_1^*\right)\jump{B_2}\\
&= \frac{1}{2}g\average{h^2}\jump{v_1} + g\average{v_1}\average{h}\jump{h} - \average{hB_1}\big(\average{v_1}\jump{B_1} + \average{B_1}\jump{v_1}+\average{v_2}\jump{B_2} + \average{B_2}\jump{v_2}\big) \\ &\qquad\qquad- \big(\average{v_1B_1} + \average{v_2B_2}\big)\jump{hB_1}- 2\average{\Delta x\vec{q}}^T{\vec{s}}_{i+\tfrac{1}{2}}.
\end{aligned}
\end{equation}

From the structure of the expanded entropy conserving constraint \eqref{eq:entropyConstraint3} we are prepared to select a discretization of the source term at the interface. The goal is to cancel the jump in $hB_1$ term. We choose that the source term discretization on cell $i$ will contribute to the interfaces $i+\tfrac{1}{2}$ and $i-\tfrac{1}{2}$. In exact arithmetic the dot product of the entropy variables with the source term is
\begin{equation}\label{EntropyTimesSource}
-\pderivative{}{x}(hB_1)\left(v_1B_1 + v_2B_2\right).
\end{equation}
From \eqref{EntropyTimesSource} we see that it will be possible to create a consistent discretization to cancel the $\jump{hB_1}$ term that remains in \eqref{eq:entropyConstraint3}. To do so we require the source term discretization
\begin{equation}\label{SourceTermDisc}
\vec{s}_i = \frac{1}{2}\left(\vec{s}_{i+\tfrac{1}{2}} + \vec{s}_{i-\tfrac{1}{2}}\right) = -\frac{1}{2}\left(
\jump{hB_1}_{i+\tfrac{1}{2}}
\begin{bmatrix}
0\\
0\\
0\\
\frac{\average{v_1B_1}}{\average{\Delta x B_1}}\\[0.1cm]
\frac{\average{v_2B_2}}{\average{\Delta x B_2}}
\end{bmatrix}_{i+\tfrac{1}{2}}
+
\jump{hB_1}_{i-\tfrac{1}{2}}
\begin{bmatrix}
0\\
0\\
0\\
\frac{\average{v_1B_1}}{\average{\Delta x B_1}}\\[0.1cm]
\frac{\average{v_2B_2}}{\average{\Delta x B_2}}
\end{bmatrix}_{i-\tfrac{1}{2}}
\right).
\end{equation}
The source term discretization \eqref{SourceTermDisc} is consistent to the source term in  \eqref{SWMHD1DForced}. For example, on a regular grid where $\Delta x_L = \Delta x_R$ we have for $\vec{s}_{i+1/2}$
\begin{equation}
-\jump{hB_1}_{i+\tfrac{1}{2}}\begin{bmatrix}
0\\
0\\
0\\
\frac{\average{v_1B_1}}{\average{\Delta x B_1}}\\[0.1cm]
\frac{\average{v_2B_2}}{\average{\Delta x B_2}}
\end{bmatrix}_{i+\tfrac{1}{2}}
=
-\frac{\jump{hB_1}_{i+\tfrac{1}{2}}}{\Delta x}
\begin{bmatrix}
0\\
0\\
0\\
\frac{\average{v_1B_1}}{\average{B_1}}\\[0.1cm]
\frac{\average{v_2B_2}}{\average{B_2}}
\end{bmatrix}_{i+\tfrac{1}{2}},
\end{equation}
where consistency is easily checked. We note that there may be computational issues if either $B_1$ or $B_2$ is zero.  However, if $B_2=0$ (for example), then the contribution to the entropy conservation constraint is also zero. Now it is possible to use the source term discretization \eqref{SourceTermDisc} to see the explicit cancellation of the problematic $\jump{hB_1}$ term at the $i+\tfrac{1}{2}$ interface:
\begin{equation}
\begin{aligned}
- \big(\average{v_1B_1} + \average{v_2B_2}\big)\jump{hB_1} - 2\average{\Delta x\vec{q}}^T{\vec{s}}_{i+\tfrac{1}{2}}&= - \big(\average{v_1B_1} + \average{v_2B_2}\big)\jump{hB_1}  \\ &\quad\quad\quad- 2\average{\Delta x\vec{q}}^T\left(-\frac{\jump{hB_1}}{2}
\begin{bmatrix}
0\\
0\\
0\\
\frac{\average{v_1B_1}}{\average{\Delta xB_1}}\\[0.1cm]
\frac{\average{v_2B_2}}{\average{\Delta xB_2}}
\end{bmatrix}\right),\\[0.1cm]
&= - \big(\average{v_1B_1} + \average{v_2B_2}\big)\jump{hB_1} \\ &\quad+ \big(\average{v_1B_1} + \average{v_2B_2}\big)\jump{hB_1}, \\[0.1cm]
&= 0.
\end{aligned}
\end{equation}

With the remaining $\jump{hB_1}$ term is removed from \eqref{eq:entropyConstraint3} we are left with the constraint on the numerical flux that only depends on the linear jump in each primitive variable
\begin{equation}\label{eq:entropyConstraint4}
\begin{aligned}
g\jump{h}f_1^*&+\left(f_2^*-\average{v_1}f_1^*\right)\jump{v_1}+\left(f_3^*-\average{v_2}f_1^*\right)\jump{v_2} + \left(f_4^*-\average{B_1}f_1^*\right)\jump{B_1}+\left(f_5^*-\average{B_2}f_1^*\right)\jump{B_2}\\
&= \frac{1}{2}g\average{h^2}\jump{v_1} + g\average{v_1}\average{h}\jump{h} - \average{hB_1}\big(\average{v_1}\jump{B_1} + \average{B_1}\jump{v_1}+\average{v_2}\jump{B_2} + \average{B_2}\jump{v_2}\big).
\end{aligned}
\end{equation}
We collect the like terms and determine a set of equations for the unknown components of the entropy conserving numerical flux $\vec{f}^{*,ec}$:
\begin{equation}\label{eq:fluxAlmostThere}
\begin{aligned}
\jump{h}:\qquad&f_1^{*,ec}= \average{h}\average{v_1} \\[0.1cm]
\jump{v_1}:\qquad&f_2^{*,ec}= \average{v_1}f_1^{*,ec} + \frac{1}{2}g\average{h^2} - \average{hB_1}\average{B_1} \\[0.1cm]
\jump{v_2}:\qquad&f_3^{*,ec}= \average{v_2}f_1^{*,ec} - \average{hB_1}\average{B_2} \\[0.1cm]
\jump{B_1}:\qquad&f_4^{*,ec}= \average{B_1}f_1^{*,ec}-\average{hB_1}\average{v_1} \\[0.1cm]
\jump{B_2}:\qquad&f_5^{*,ec}= \average{B_2}f_1^{*,ec}-\average{hB_1}\average{v_2}
\end{aligned}
\end{equation}
The equations \eqref{eq:fluxAlmostThere} are easily solved to determine the form of the entropy conserving flux to be
\begin{equation}\label{eq:SWMHDECFlux}
\vec{f}^{*,ec} = \begin{bmatrix}
\average{h}\average{v_1} \\[0.1cm]
\average{h}\average{v_1}^2 + \frac{1}{2}g\average{h^2} - \average{hB_1}\average{B_1}\\[0.1cm]
\average{h}\average{v_1}\average{v_2} - \average{hB_1}\average{B_2} \\[0.1cm]
\average{h}\average{v_1}\average{B_1}-\average{hB_1}\average{v_1} \\[0.1cm]
\average{h}\average{v_1}\average{B_2}-\average{hB_1}\average{v_2}
\end{bmatrix}.
\end{equation}
Interestingly, to guarantee the conservation of entropy we had to relax the divergence-free condition to $\partial_x(hB_1)\approx 0$. This assumption manifests itself in the fourth component of $\vec{f}^{*,ec}$, which is nonzero.  Essentially the fourth term in $\vec{f}^{*,ec}$ is the commutator of multiplication and averaging of $hB_1$. 
\end{proof}

\begin{rem}
If the flow occurs without the presence of a magnetic field then the SWMHD model becomes the traditional shallow water (SW) equations. The structure of $\vec{f}^{*,ec}$ \eqref{eq:SWMHDECFlux} can be separated into a shallow water component and magnetic field component, i.e.,
\begin{equation}
\vec{f}^{*,ec} = \begin{bmatrix}
\average{h}\average{v_1} \\[0.1cm]
\average{h}\average{v_1}^2 + \frac{1}{2}g\average{h^2} \\[0.1cm]
\average{h}\average{v_1}\average{v_2}  \\[0.1cm]
0\\[0.1cm]
0
\end{bmatrix}
+
\begin{bmatrix}
0 \\[0.1cm]
- \average{hB_1}\average{B_1}\\[0.1cm]
- \average{hB_1}\average{B_2} \\[0.1cm]
\average{h}\average{v_1}\average{B_1}-\average{hB_1}\average{v_1} \\[0.1cm]
\average{h}\average{v_1}\average{B_2}-\average{hB_1}\average{v_2}
\end{bmatrix}.
\end{equation}
Thus, if the magnetic field is zero we find 
\begin{equation}
\vec{f}^{*,ec} = \begin{bmatrix}
\average{h}\average{v_1} \\[0.1cm]
\average{h}\average{v_1}^2 + \frac{1}{2}g\average{h^2} \\[0.1cm]
\average{h}\average{v_1}\average{v_2}  \\[0.1cm]
0\\[0.1cm]
0
\end{bmatrix}
+
\begin{bmatrix}
0 \\[0.1cm]
0\\[0.1cm]
0 \\[0.1cm]
0 \\[0.1cm]
0
\end{bmatrix},
\end{equation}
and $\vec{f}^{*,ec}$ becomes the appropriate entropy conserving flux for the SW equations described in \cite{fjordholm2011,gassner2014}.
\end{rem}

\begin{rem}
We have considered a computational problem with a flat bottom topography. So it is a natural question for a shallow water type model if the approximation remains well-balanced, i.e., recovers a constant fluid height solution in the presence of a non-constant bottom topography \cite{winters2014}. Previous work by Fjordholm et al. \cite{fjordholm2011} provides a consistent, well-balanced, and entropy conserving treatment of a non-constant (and even discontinuous) bottom topography $b$. For example in the $x-$direction
\begin{equation}\label{botSource}
\left(gh\pderivative{b}{x}\right)_i \approx \frac{g}{2\Delta x_i}\left(\average{h}_{i+\tfrac{1}{2}}\jump{b}_{i+\tfrac{1}{2}} + \average{h}_{i-\tfrac{1}{2}}\jump{b}_{i-\tfrac{1}{2}}\right).
\end{equation}
In the presence of a non-constant bottom topography one would add the discretized source term \eqref{botSource} to the $x$ component of the momentum equation of the SWMHD model. Note that the treatment of the bottom topography source term by Fjordholm et al. in \eqref{botSource} is analogous to our treatment of the Janhunen source term for SWMHD \eqref{SourceTermDisc}. 
\end{rem}

\subsection{Dissipation Terms for an Entropy Stable Flux}\label{Sec:StableFlux}

To create an entropy stable numerical flux function we use the entropy conserving flux \eqref{eq:SWMHDECFlux} as a base and subtract some form of numerical dissipation, e.g,
\begin{equation}\label{dissipation}
\vec{f}^* = \vec{f}^{*,ec} - \frac{1}{2}\matrix{D}\jump{\vec{u}},
\end{equation}
where $\matrix{D}$ is a dissipation matrix. 

Our goal to guarantee the entropy stability of the approximation is to reformulate the dissipation term \eqref{dissipation} such that it is guaranteed to be positive. To do so we will consider the eigenstructure of the flux Jacobian $\matrix{A}$ altered by the Powell source term 
\begin{equation}\label{PowellSource}
\vec{s}_{Powell} = -\pderivative{}{x}(hB_1)\begin{bmatrix}
0 \\
B_1 \\
B_2 \\
v_1 \\
v_2
\end{bmatrix}.
\end{equation}
It is important to note that we use the Janhunen source term \eqref{SWMHD1DForced} to {\color{black}{derive an}} entropy {\color{black}{conservative numerical flux function}} in Sec. \ref{Sec:ConservingFlux}. However, to design an entropy stable approximation we require {\color{black}{that}} the {\color{black}{eigendecomposition of the}} flux Jacobian matrix {\color{black}{can be related to the entropy Jacobian \eqref{entropyJacobian}. This particular scaling, first examined by Merriam \cite{merriam1989} and explored more thoroughly by Barth \cite{barth99}, requires that the PDE system is symmetrizable.}} Previous {\color{black}{analysis of}} the SWMHD {\color{black}{system by Dellar}} \cite{dellar2002} and ideal MHD equations \cite{barth99,godunov1972} have demonstrated that the Powell source term is necessary to restore a symmetric MHD system. We {\color{black}{reiterate}} that the altered flux Jacobian is used only to derive the dissipation term. {\color{black}{Just as Lax-Friedrichs differs from Roe in the construction of a dissipation term, we use the Powell source term only to build our dissipation term.}} Thus, we do not introduce any inconsistency in the previous {\color{black}{entropy conserving flux}} derivations. {\color{black}{We demonstrate with numerical tests in Sec. \ref{ESRiemann} that the newly developed entropy stable flux functions are competitive with the well known Roe flux.}}

We denote the flux Jacobian modified with the Powell source term \eqref{PowellSource} as
\begin{equation}\label{alteredFluxJacobian}
\doublehat{\matrix{A}} = \begin{bmatrix} 0 & 1 & 0 & 0 & 0  \\ 
gh - v_1^2 + B_1^2 & 2v_1 & 0 & -B_1 & 0 \\
-v_1v_2+B_1B_2 & v_2 & v_1 & 0 & -B_1 \\
0 & 0 & 0 & v_1 & 0 \\
v_2B_1 - v_1B_2 & B_2 & -B_1 & 0 & v_1 \\
\end{bmatrix},
\end{equation}
equipped with a full set of eigenvalues
\begin{equation}\label{eigenvalues}
\lambda_1 = v_1-c_g,\quad\lambda_2 = v_1-B_1,\quad\lambda_3 = v_1,\quad\lambda_4 = v_1+B_1,\quad\lambda_5 = v_1+c_g,
\end{equation}
and right eigenvectors 
\begin{equation}\label{rightEV}
\doublehat{\matrix{R}} = \begin{bmatrix}
1 & 0 & 1 & 0 & 1 \\[0.1cm]
v_1-c_g & 0 & v_1 & 0 & v_1+c_g \\[0.1cm]
v_2 & 1 & v_2 & 1 & v_2 \\[0.1cm]
0 & 0 & \frac{c_g^2}{B_1} & 0 & 0 \\[0.1cm]
B_2 & 1 & B_2 & -1 & B_2 
\end{bmatrix},
\quad
\doublehat{\matrix{L}} = \doublehat{\matrix{R}}^{-1},
\end{equation}
where $c_g^2 = gh+B_1^2$ is the magnetogravity wave speed.

Now that we have a symmetrizable matrix $\doublehat{\matrix{A}}$ we utilize a previous result from Barth \cite{barth99} which provides a systematic approach to restructure a general eigenvalue problem to a symmetric eigenvalue problem. To do so we rescale the right eigenvectors of an eigendecomposition with respect to a right symmetrizer matrix in the following way:
\begin{lem}(Eigenvector Scaling)
Let $\matrix{A}\in\mathbb{R}^{n\times n}$ be an arbitrary diagonalizable matrix and $S$ the set of all right symmetrizers:
\begin{equation}
S=\left\{\matrix{B}\in\mathbb{R}^{n\times n}\,\big|\;\matrix{B}\;is\; s.p.d,\;\;\matrix{AB} = (\matrix{AB})^T\right\}.
\end{equation}
Further, let $\matrix{R}\in\mathbb{R}^{n\times n}$ denote the right eigenvector matrix which diagonalizes $\matrix{A}$, i.e., $\matrix{A}=\matrix{R}\boldsymbol\Lambda\matrix{R}^{-1}$, with $r$ distinct eigenvalues. Then for each $\matrix{B}\in S$ there exists a symmetric block diagonal matrix $\matrix{T}$ that block scales columns of $\matrix{R}$, $\tilde{R} = \matrix{RT}$, such that
\begin{equation}
\matrix{B}=\widetilde{\matrix{R}}\widetilde{\matrix{R}}^T,\; \matrix{A}=\widetilde{\matrix{R}}\boldsymbol\Lambda\widetilde{\matrix{R}}^{-1},
\end{equation}
which implies
\begin{equation}
\matrix{AB}=\widetilde{\matrix{R}}\boldsymbol\Lambda\widetilde{\matrix{R}}^{T}.
\end{equation}
\end{lem}

\begin{proof} The proof of the eigenvector scaling lemma is given in \cite{barth99}.\end{proof}

\begin{thm} (Entropy Stable 1 (ES1)) If we apply the diagonal scaling matrix
\begin{equation}\label{scalingMatrixInTheorem}
\matrix{T} = diag\left(\frac{c}{c_g\sqrt{2g}}\,,\,\frac{c}{\sqrt{2g}}\,,\,\frac{B_1}{c_g\sqrt{g}}\,,\,\frac{c}{\sqrt{2g}}\,,\,\frac{c}{c_g\sqrt{2g}}\right),
\end{equation}
to the matrix of right eigenvectors $\doublehat{\matrix{R}}$ \eqref{rightEV}, then we obtain the Merriam identity \cite{merriam1989} (Eq. 7.3.1 pg. 77) 
\begin{equation}\label{MerriamIdentity}
\matrix{H} = \widetilde{\matrix{R}}\widetilde{\matrix{R}}^T = \left(\doublehat{\matrix{R}}\matrix{T}\right) \left(\doublehat{\matrix{R}}\matrix{T}\right)^T = \doublehat{\matrix{R}}\matrix{S}\doublehat{\matrix{R}}^T,
\end{equation}
that relates the right eigenvectors of $\doublehat{\matrix{A}}$ to the entropy Jacobian matrix \eqref{entropyJacobian}. For convenience, we introduce the diagonal scaling matrix $\matrix{S}=\matrix{T}\,^2$ in \eqref{MerriamIdentity}. We then have the guaranteed entropy stable flux interface contribution
\begin{equation}\label{minimalDiss}
\vec{f}^{*,ES1} = \vec{f}^{*,ec}-\frac{1}{2}\matrix{D}\jump{\vec{u}}=\vec{f}^{*,ec} - \frac{1}{2} \doublehat{\matrix{R}}|\doublehat{\boldsymbol\Lambda}|\matrix{S}\doublehat{\matrix{R}}^T\jump{\vec{q}}.
\end{equation}
\end{thm}

\begin{proof}
We define the dissipation term in the numerical flux \eqref{dissipation} to be
\begin{equation}\label{dissTerm1}
-\frac{1}{2}\matrix{D}\jump{\vec{u}} = -\frac{1}{2} |\doublehat{\matrix{A}}|\jump{\vec{u}} = -\frac{1}{2} \doublehat{\matrix{R}}|\doublehat{\boldsymbol\Lambda}|\doublehat{\matrix{L}}\jump{\vec{u}},
\end{equation}
where the eigendecomposition of $\doublehat{\matrix{A}} = \doublehat{\matrix{R}}\doublehat{\boldsymbol\Lambda}\doublehat{\matrix{L}}$ is given by \eqref{eigenvalues} and \eqref{rightEV}. We define entropy stability to mean the approximation guarantees that the entropy within the system is a decreasing function, satisfying the following inequality
\begin{equation}
\pderivative{U}{t} + \pderivative{F}{x} - \vec{q}^T\vec{s} \leq 0.
\end{equation}
From the previously computed discrete entropy update \eqref{TotalUpdate} the total entropy within an element (now including the dissipative term \eqref{dissTerm1}) is
\begin{equation}\label{TotalUpdate2}
\begin{aligned}
\pderivative{}{t}(\Delta x_L U_L + \Delta x_R U_R) &= -\jump{\vec{q}\cdot\vec{f}} + \jump{\vec{q}}^T\vec{f}^*  + 2\average{\Delta x\vec{q}}{\vec{s}}_{i+\tfrac{1}{2}}, \\
&= -\jump{\vec{q}\cdot\vec{f}} + \jump{\vec{q}}^T\vec{f}^{*,ec} - \frac{1}{2}\jump{\vec{q}}^T \doublehat{\matrix{R}}|\doublehat{\boldsymbol\Lambda}|\doublehat{\matrix{L}}\jump{\vec{u}} + 2\average{\Delta x\vec{q}}{\vec{s}}_{i+\tfrac{1}{2}}, \\
\pderivative{}{t}(\Delta x_L U_L + \Delta x_R U_R) + \jump{F} &= - \frac{1}{2}\jump{\vec{q}}^T \doublehat{\matrix{R}}|\doublehat{\boldsymbol\Lambda}|\doublehat{\matrix{L}}\jump{\vec{u}},
\end{aligned}
\end{equation}
from the design of the entropy conserving flux $\vec{f}^{*,ec}$. To ensure entropy stability, we must guarantee that the RHS term in \eqref{TotalUpdate2} is non-positive. Unfortunately, it was shown by Barth \cite{barth99} that the term 
\begin{equation}\label{RHSEntropy}
- \frac{1}{2}\jump{\vec{q}}^T \doublehat{\matrix{R}}|\doublehat{\boldsymbol\Lambda}|\doublehat{\matrix{L}}\jump{\vec{u}},
\end{equation}
may become positive in the presence of very strong shocks. However, we know from entropy symmetrization theory, e.g \cite{barth99,merriam1989}, that the entropy Jacobian $\matrix{H}$, given by \eqref{entropyJacobian}, is a right symmetrizer for the flux Jacobian that incorporates the Powell source term $\doublehat{\matrix{A}}$. Therefore, with the proper scaling matrix $\matrix{T}$ we acquire the Merriam identity 
\begin{equation}\label{MerriamIdentity2}
\matrix{H} = \widetilde{\matrix{R}}\widetilde{\matrix{R}}^T = \left(\doublehat{\matrix{R}}\matrix{T}\right) \left(\doublehat{\matrix{R}}\matrix{T}\right)^T = \doublehat{\matrix{R}}\matrix{S}\doublehat{\matrix{R}}^T.
\end{equation}
The rescaling of the right eigenvectors of $\doublehat{\matrix{A}}$ to satisfy the Merriam identity \eqref{MerriamIdentity} is sufficient to guarantee the negativity of \eqref{RHSEntropy}. We see from \eqref{RHSEntropy} and \eqref{MerriamIdentity2}
\begin{equation}\label{signSatified}
\begin{aligned}
-\frac{1}{2}\jump{\vec{q}}^T \doublehat{\matrix{R}}|\doublehat{\boldsymbol\Lambda}|\doublehat{\matrix{L}}\jump{\vec{u}} &\simeq - \frac{1}{2}\jump{\vec{q}}^T \doublehat{\matrix{R}}|\doublehat{\boldsymbol\Lambda}|\doublehat{\matrix{L}}\vec{u}_{\vec{q}}\jump{\vec{q}},  \\
&= - \frac{1}{2}\jump{\vec{q}}^T \doublehat{\matrix{R}}|\doublehat{\boldsymbol\Lambda}|\doublehat{\matrix{L}}\matrix{H}\jump{\vec{q}}, \\
&=- \frac{1}{2}\jump{\vec{q}}^T \doublehat{\matrix{R}}|\doublehat{\boldsymbol\Lambda}|\doublehat{\matrix{L}}\left(\doublehat{\matrix{R}}\matrix{S}\doublehat{\matrix{R}}^T\right)\jump{\vec{q}}, \\
&=- \frac{1}{2}\jump{\vec{q}}^T \doublehat{\matrix{R}}|\doublehat{\boldsymbol\Lambda}|\matrix{S}\doublehat{\matrix{R}}^T\jump{\vec{q}},
\end{aligned}
\end{equation}
where we used that $\doublehat{\matrix{L}}=\doublehat{\matrix{R}}^{-1}$. So with the appropriate diagonal scaling matrix $\matrix{S}$ we have shown that \eqref{signSatified} is guaranteed negative because the product is a quadratic form scaled by a negative. We use the right eigenvectors from \eqref{rightEV}, the constraint \eqref{MerriamIdentity}, and straightforward algebraic manipulation to determine the diagonal scaling matrix
\begin{equation}\label{scalingMatrix}
\matrix{T} = diag\left(\frac{c}{c_g\sqrt{2g}}\,,\,\frac{c}{\sqrt{2g}}\,,\,\frac{B_1}{c_g\sqrt{g}}\,,\,\frac{c}{\sqrt{2g}}\,,\,\frac{c}{c_g\sqrt{2g}}\right),
\end{equation}
where $c^2 = gh$ and $c_g^2 = gh+B_1^2$ are the wave speed and magnetogravity wave speed respectively.
\end{proof}

\begin{rem}
There are other possible, negativity guaranteeing (but more dissipative) choices for the dissipation term \eqref{dissipation}. For example, if we make the simple choice of dissipation matrix to be
\begin{equation}\label{LFDMat}
\matrix{D} = |\lambda_{max}|\matrix{I},
\end{equation}
where $\lambda_{max}$ is the largest eigenvalue of the system from \eqref{eigenvalues} and $\matrix{I}$ is the identity matrix, then we obtain a local Lax-Friedrichs type interface stabilization
\begin{equation}\label{LFDiss}
\begin{aligned}
\vec{f}^{*,ES2} &=  \vec{f}^{*,ec} - \frac{1}{2}|\lambda_{max}|\matrix{I}\jump{\vec{u}}, \\
&=  \vec{f}^{*,ec} - \frac{1}{2}|\lambda_{max}|\matrix{H}\jump{\vec{q}}, \\
&= \vec{f}^{*,ec} - \frac{1}{2}\doublehat{\matrix{R}}|\lambda_{max}|\matrix{S}\doublehat{\matrix{R}}^T\jump{\vec{q}},\\
\end{aligned}
\end{equation} 
where, again, we use the Merriam identity \eqref{MerriamIdentity} for the entropy Jacobian $\matrix{H}$.
\end{rem}

\section{Numerical Results}\label{NumericalResults}

In this section, we numerically verify the theoretical findings for the entropy conserving and entropy stable approximations for the SWMHD equations. To integrate the semi-discrete formulation in time we use a low storage five-stage, fourth-order accurate Runge-Kutta time integrator of Carpenter and Kennedy \cite{Carpenter&Kennedy:1994}. First, in Sec. \ref{EOC}, we consider a test problem with a known analytical solution to demonstrate the accuracy of the entropy conserving method as well as the optimal and Lax-Friedrichs stabilized formulations. Next, Sec. \ref{EntropyCons} demonstrates the entropy conservation of the approximation for a strong Riemann problem. In Sec. \ref{ECRiemann} we demonstrate the computed solution of a strong Riemann problem for the entropy conserving method. This solution will exhibit significant oscillations in shocked regions. Finally, in Sec. \ref{ESRiemann} we compare the two entropy stable approximations \eqref{minimalDiss} and \eqref{LFDiss} against a high-resolution approximation comparable to that presented in the literature \cite{desterck2001,rossmanith2002}.

\subsection{Convergence}\label{EOC}

For the convergence test, we switch to the manufactured solution technique and generate a smooth and periodic solution 
\begin{equation}
\begin{bmatrix}
h(x,t)\\
hv_1(x,t)\\
hv_2(x,t)\\
hB_1(x,t)\\
hB_2(x,t)\\
\end{bmatrix}
=%
\begin{bmatrix}
2+\sin{(2\pi\,(x-t))}\\
2+\sin{(2\pi\,(x-t))}\\
2+\sin{(2\pi\,(x-t))}\\
1\\
4+2\,\sin{(2\pi\,(x-t))}
\end{bmatrix}
\end{equation}
with an additional analytic source term on the right hand side
\begin{equation}
\begin{bmatrix}
s_1(x,t)\\
s_2(x,t)\\
s_3(x,t)\\
s_4(x,t)\\
s_5(x,t)
\end{bmatrix}
=%
\begin{bmatrix}
0\\
h_x(x,t)\left(g\,h(x,t)+\frac{1}{h(x,t)^2}\right)\\
0\\
0\\
0
\end{bmatrix}
\end{equation}
on the domain $\Omega=[-1,1]$ with periodic boundary conditions and final time $T=2$. We select the time step for the RK method small enough such that the error in the approximation is dominated by the error in the spatial discretizations. For all computations, a regular grid is chosen according to the number of grid cells listed in the tables below. Except for the entropy conserving scheme, where we also test a stretched mesh with a fixed ratio of 
\begin{equation}\label{gridRatio}
\frac{\Delta x_{max}}{\Delta x_{min}}=4.
\end{equation}

First, we test the entropy conserving scheme on a regular grid. The $L_2$-errors for all conserved quantities are shown in Tbl.~\ref{tab:EOC_EC_regular}. It is interesting to note, that for the specific regular grid used in the numerical experiment, the entropy conserving scheme achieves second order accuracy. We note that the average accuracy is hovering around $1.7$, however looking closely at the finest grid results, it is clear that the experimental order of convergence is close to $2$. 
However, the higher order convergence for the finite volume scheme is an effect of approximating the solution on a regular grid. The scheme drops to first order accuracy when an irregular grid is chosen as shown in Tbl.~\ref{tab:EOC_EC_irregular}, where we stretch the grid with a constant factor of \eqref{gridRatio}.

\begin{table}[!ht]
\begin{center}
\begin{tabular}{|c||c|c|c|c|c|}
\hline
$\#$ elements & $L_2$ error $h$ & $L_2$ error $hv_1$ & $L_2$ error $hv_2$ & $L_2$ error $hB_1$ & $L_2$ error $hB_2$ \\
\hline\hline
50   & 2.05E-2  & 2.12E-2 & 2.91E-2   & 3.77E-2   & 1.62E-1\\
\hline
100 &  8.78E-3  & 9.14E-3 & 9.45E-3   &  1.90E-2   & 3.08E-2\\ 
\hline
200 &  2.45E-3  & 2.50E-3 & 2.59E-3   &  3.12E-3   & 7.22E-3\\ 
\hline
400 &  6.17E-4  & 6.26E-4 & 6.37E-4   &  7.85E-4   & 1.79E-3\\ 
\hline\hline
avg EOC & \bf 1.69 & \bf 1.69 & \bf 1.84 & \bf 1.86 & \bf 2.17\\
\hline
\end{tabular}
\end{center}
\caption{Experimental order of convergence EOC for \textit{entropy conserving} scheme on a regular grid}
\label{tab:EOC_EC_regular}
\end{table}%

\begin{table}[!ht]
\begin{center}
\begin{tabular}{|c||c|c|c|c|c|}
\hline
$\#$ elements & $L_2$ error $h$ & $L_2$ error $hv_1$ & $L_2$ error $hv_2$ & $L_2$ error $hB_1$ & $L_2$ error $hB_2$ \\
\hline\hline
50   & 1.44E-1  & 2.50E-1 & 5.90E-1   & 1.64E-1   & 1.05E-0\\
\hline
100 &  5.57E-2  & 9.66E-2 & 2.11E-1   &  5.57E-2   & 2.62E-1\\ 
\hline
200 &  2.73E-2  & 4.57E-2 & 1.02E-2   &  2.47E-2   & 1.22E-1\\ 
\hline
400 &  1.34E-2  & 2.26E-2 & 5.45E-2   &  1.20E-2   & 6.22E-2\\ 
\hline\hline
avg EOC & \bf 1.14 & \bf 1.16 & \bf 1.15 & \bf 1.26 & \bf 1.36\\
\hline
\end{tabular}
\end{center}
\caption{Experimental order of convergence EOC for \textit{entropy conserving} scheme on an irregular grid with a stretching factor of \eqref{gridRatio}.}
\label{tab:EOC_EC_irregular}
\end{table}%

Next, we demonstrate the convergence of the two entropy stable finite volume schemes. The convergence results for the ES1 method are shown for the regular grid test in Tbl.~\ref{tab:EOC_ROE_regular} and the irregular grid test in Tbl.~\ref{tab:EOC_ROE_irregular}, where we see the average experimental convergence rate for either the regular or irregular grids are both first order accurate. The convergence behavior for the ES2 scheme on both grids fluctuates. Thus, we decided to run more tests on even finer grids to demonstrate that, again, the experimental order of convergence for regular or irregular grids is in the range of $0.7 -1.36$.

\begin{table}[!ht]
\begin{center}
\begin{tabular}{|c||c|c|c|c|c|}
\hline
$\#$ elements & $L_2$ error $h$ & $L_2$ error $hv_1$ & $L_2$ error $hv_2$ & $L_2$ error $hB_1$ & $L_2$ error $hB_2$ \\
\hline\hline
50   & 6.24E-2  & 1.87E-1 & 6.44E-2   & 6.84E-3   & 1.16E-1\\
\hline
100 &  2.91E-2  & 8.49E-2 & 2.94E-2   &  3.20E-3   & 5.33E-2\\ 
\hline
200 &  1.25E-2  & 3.38E-2 & 1.25E-2   &  1.33E-3   & 2.28E-2\\ 
\hline
400 &  5.79E-3  & 1.27E-2 & 5.78E-3   &  4.72E-4   & 1.10E-2\\ 
\hline\hline
avg EOC & \bf 1.14 & \bf 1.29 & \bf 1.16 & \bf 1.29 & \bf 1.13\\
\hline
\end{tabular}
\end{center}
\caption{Experimental order of convergence EOC for \textit{ES1 scheme} \eqref{minimalDiss} on a regular grid.}
\label{tab:EOC_ROE_regular}
\end{table}%

\begin{table}[!ht]
\begin{center}
\begin{tabular}{|c||c|c|c|c|c|}
\hline
$\#$ elements & $L_2$ error $h$ & $L_2$ error $hv_1$ & $L_2$ error $hv_2$ & $L_2$ error $hB_1$ & $L_2$ error $hB_2$ \\
\hline\hline
50   & 8.29E-2  & 2.72E-1 & 8.81E-2   & 9.06E-3   & 1.66E-1\\
\hline
100 &  4.25E-2  & 1.40E-1 & 4.40E-2   &  4.20E-3   & 8.33E-2\\ 
\hline
200 &  2.04E-2  & 6.66E-2 & 2.07E-2   &  1.82E-3   &  3.96E-2\\ 
\hline
400 &  1.02E-2  & 3.147E-2 & 1.02E-2   &  6.66E-4   &  1.99E-2\\ 
\hline\hline
avg EOC & \bf 1.01 & \bf 1.04 & \bf 1.04 & \bf 1.26 & \bf 1.02\\
\hline
\end{tabular}
\end{center}
\caption{Experimental order of convergence EOC for \textit{ES1 scheme} \eqref{minimalDiss} on an irregular grid with a stretching factor of \eqref{gridRatio}.}
\label{tab:EOC_ROE_irregular}
\end{table}%

\begin{table}[!ht]
\begin{center}
\begin{tabular}{|c||c|c|c|c|c|}
\hline
$\#$ elements & $L_2$ error $h$ & $L_2$ error $hv_1$ & $L_2$ error $hv_2$ & $L_2$ error $hB_1$ & $L_2$ error $hB_2$ \\
\hline\hline
50   & 2.29E-2  & 2.19E-1 & 2.48E-2   & 2.91E-3   & 4.80E-2\\
\hline
100 &  1.02E-2  & 9.92E-2 & 1.14E-2   &  1.41E-3   & 2.13E-2\\ 
\hline
200 &  9.14E-3  & 4.03E-2 & 9.37E-3   &  6.74E-4   & 1.85E-2\\ 
\hline
400 &  7.06E-3  & 1.62E-2 & 7.08E-3   &  2.92E-4   & 1.42E-2\\ 
\hline
800 &  4.44E-3  & 7.06E-3 & 4.44E-3   &  1.08E-4   & 8.89E-3\\ 
\hline
1600 &  2.49E-3  & 3.32E-3 & 2.49E-3   &  3.48E-5   & 4.99E-3\\ 
\hline
3200 &  1.32E-3  & 1.62E-3 & 1.32E-3   &  1.01E-5   & 2.64E-3\\ 
\hline\hline
avg EOC & \bf 1.24 & \bf 1.18 & \bf 0.70 & \bf 1.36 & \bf 0.70\\
\hline
\end{tabular}
\end{center}
\caption{Experimental order of convergence EOC for \textit{ES2 scheme} \eqref{LFDiss} on a regular grid.}
\label{tab:EOC_LF_regular}
\end{table}%

\begin{table}[!ht]
\begin{center}
\begin{tabular}{|c||c|c|c|c|c|}
\hline
$\#$ elements & $L_2$ error $h$ & $L_2$ error $hv_1$ & $L_2$ error $hv_2$ & $L_2$ error $hB_1$ & $L_2$ error $hB_2$ \\
\hline\hline
50   & 6.65E-2  & 3.13E-1 & 6.74E-2   & 3.66E-3   & 1.38E-1\\
\hline
100 &  3.39E-2  & 1.58E-1 & 3.43E-2   &  1.81E-3   & 6.98E-2\\ 
\hline
200 &  1.97E-2  & 7.29E-2 & 1.97E-2   &  8.82E-4   & 3.96E-2\\ 
\hline
400 &  1.26E-2  & 3.37E-2 & 1.26E-2   &  3.94E-4   & 2.51E-2\\ 
\hline
800 &  7.83E-3  & 1.64E-2 & 7.80E-3   &  1.51E-4   & 1.56E-2\\ 
\hline
1600 &  4.57E-3  & 8.22E-3 & 4.57E-3   &  4.97E-5   & 9.12E-3\\ 
\hline
3200 &  2.54E-3  & 4.15E-3 & 2.54E-3   &  1.48E-5   & 5.07E-3\\ 
\hline\hline
avg EOC & \bf 0.78 & \bf 1.04 & \bf 0.79 & \bf 1.33 & \bf 0.79\\
\hline
\end{tabular}
\end{center}
\caption{Experimental order of convergence EOC for \textit{ES2 scheme} on an irregular grid with a stretching factor of \eqref{gridRatio}.}
\label{tab:EOC_LF_irregular}
\end{table}%

\subsection{Mass, Momentum, and Entropy Conservation}\label{EntropyCons}

By design we know the finite volume scheme for the SWMHD equations with the Janhunen source term \eqref{SWMHD1DForced} will conserve mass and momentum. From the derivations in \ref{Sec:ConservingFlux} we know that the scheme also exactly preserves the entropy on a general grid, if we use the newly designed flux $\vec{f}^{*,ec}$ \eqref{eq:SWMHDECFlux}. We measure the change in any of the conservative variables with
\begin{equation}
\label{eq:errortotalenergy}
\Delta e(T) := |e_{int}(t=0) - e_{int}(T)|,
\end{equation}
where, for example, $e_{int}$ is the entropy function
\begin{equation}\label{totalEnergy}
U =  \frac{1}{2}\left(gh^2+hv_1^2+hv_2^2+hB_1^2+hB_2^2\right)
\end{equation} 
integrated with a Gauss-Lobatto quadrature rule over the whole domain.

To demonstrate the conservative properties we consider the one dimensional strong Riemann problem taken from \cite{desterck2001} with the initial condition
\begin{equation}\label{StrongRiemann}
\begin{bmatrix}h\\v_1\\v_2\\B_1\\B_2\end{bmatrix} = \left\{
\begin{array}{lc}
\left[1,0,0,1,0\right]^T, & \textrm{if}\quad x \leq 0, \\ 
\left[2,0,0,0.5,1\right]^T, & \textrm{if}\quad x > 0,
\end{array}
\right.
\end{equation}
on the domain $\Omega = [-1,1]$ and periodic boundary conditions. We compute the error in the change of each conservative variable for three values of the Courant-Friedrichs-Lewy (CFL) number: 1.0, 0.1, and 0.01. For each simulation the entropy conserving finite volume scheme used 100 regular grid cells for the results in Tbl. \ref{tab:conservation} and 100 irregular grid cells in Tbl. \ref{tab:conservation_irregulargrid}. The final time of each simulation was $T=0.4$.
 
We see that for each value of the CFL number we obtain errors in the mass and momentum on the order of machine precision. Also, as expected, the addition of the Janhunen source term to enforce the divergence condition causes us to lose conservation of the quantities $h\vec{B}$. Interestingly, in the error for the change in entropy we see the dissipative influence of the time integrator. The dissipation introduced by the temporal discretization can be significantly reduced if we shrink the time step. We see that the error of the entropy can be lowered to single or double machine precision by decreasing the CFL number, and hence the time step.

It is also possible to see the temporal accuracy in Tbl. \ref{tab:conservation} and Tbl. \ref{tab:conservation_irregulargrid}. If we shrink the time step by a factor ten we see that the error in the entropy shrinks by a little over a factor of $10^4$, as we expect for a fourth order time integrator. 

\begin{table}[!ht]
\begin{center}
\begin{tabular}{|c||c|c|c|c|c|c|}
\hline
$CFL$ & $h$  & $h\,v_1$ & $h\,v_2$& $h\,B_1$ & $h\,B_2$ & $U$\\
\hline\hline
$1.0$  & $-2.22E-16$ & $ 3.53E-14$ & $-5.06E-14$ & $ -2.82E-04$ & $ 6.61E-04$ & $-7.38E-05$\\
\hline
$0.1$ & $4.44E-16$ & $8.28E-14$ & $-6.99E-14$ & $-2.82E-04$ & $6.64E-04$ & $-1.54E-09$\\
\hline
$0.01$ & $-9.99E-16$ & $9.66E-14$ & $-7.48E-14$ & $-2.82E-04$ & $6.64E-04$ & $-9.20E-14$\\
\hline
\end{tabular}
\end{center}
\caption{Conservation errors (integrated over the whole domain) of the entropy conserving approximation of the Riemann problem \eqref{StrongRiemann} for different CFL numbers, final time $T=0.4$, and $100$ regular grid cells.}
\label{tab:conservation}
\end{table}%

\begin{table}[!ht]
\begin{center}
\begin{tabular}{|c||c|c|c|c|c|c|}
\hline
$CFL$ & $h$  & $h\,v_1$ & $h\,v_2$& $h\,B_1$ & $h\,B_2$ & $U$ \\
\hline\hline
$1.0$  & $0.0$ & $ 0.0$ & $4.51E-17$ & $ -9.67E-04$ & $ 2.47E-03$ & $-2.05E-05$\\
\hline
$0.1$ & $-3.33E-16$ & $-1.04E-17$ & $-2.57E-14$ & $-9.66E-04$ & $2.47E-03$ & $-3.74E-10$\\
\hline
$0.01$ & $-3.44E-15$ & $3.40E-16$ & $2.64E-16$ & $-9.66E-04$ & $2.47E-03$ & $-2.11E-14$\\
\hline
\end{tabular}
\end{center}
\caption{Conservation errors (integrated over the whole domain) for different CFL numbers with end time $t=0.4$ on an irregular grid with a stretching factor of \eqref{gridRatio}.}
\label{tab:conservation_irregulargrid}
\end{table}%

\subsection{Entropy Conserving Riemann Problem}\label{ECRiemann}

We now apply the entropy conserving (EC) scheme to the strong Riemann problem \eqref{StrongRiemann}, except we choose inflow/outflow type boundary conditions. The computation is performed on 100 regular grid cells with CFL number 0.1 up to a final time of $T=0.4$. To create a reference solution we use the ES1 scheme on 5000 grid cells. We present the solution computed with the EC finite volume method against the reference solution in Fig. \ref{fig:EC}. 

The results for the computed fluid height $h$, velocities $v_1$, $v_2$, and magnetic field $B_1$, $B_2$ respectively in Fig. \ref{fig:EC} show that the EC scheme computes the rarefactions and the shocks in the solution quite accurately, but at the expense of large post-shock oscillations. These oscillations are expected as energy must be dissipated across the shock but the EC scheme is basically dissipation free except for the influence of the time integrator (as we previously noted in Sec. \ref{EntropyCons}).
\begin{figure}[!ht]
\begin{center}
{
\includegraphics[scale=0.575]{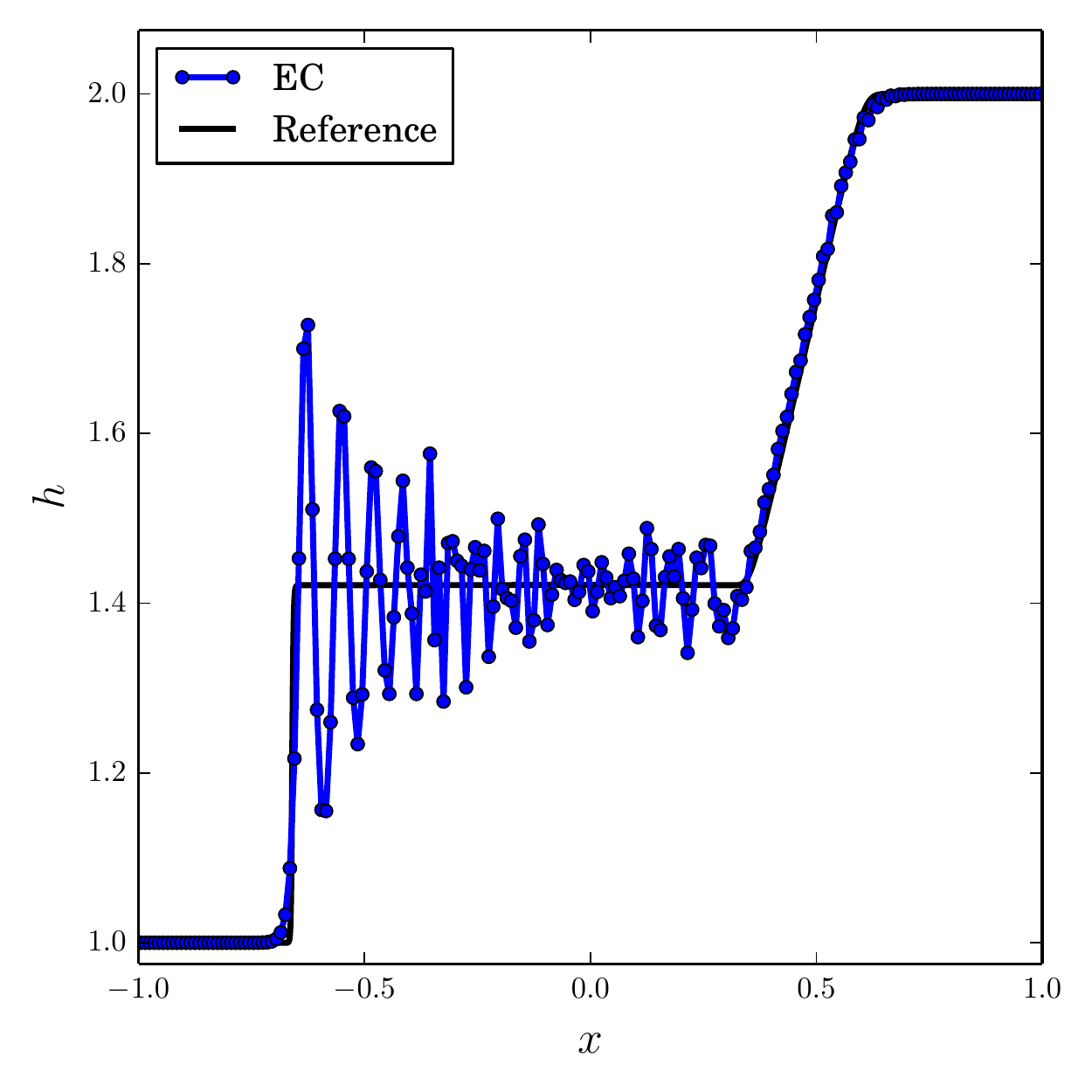}
}
\\
{
\includegraphics[scale=0.575]{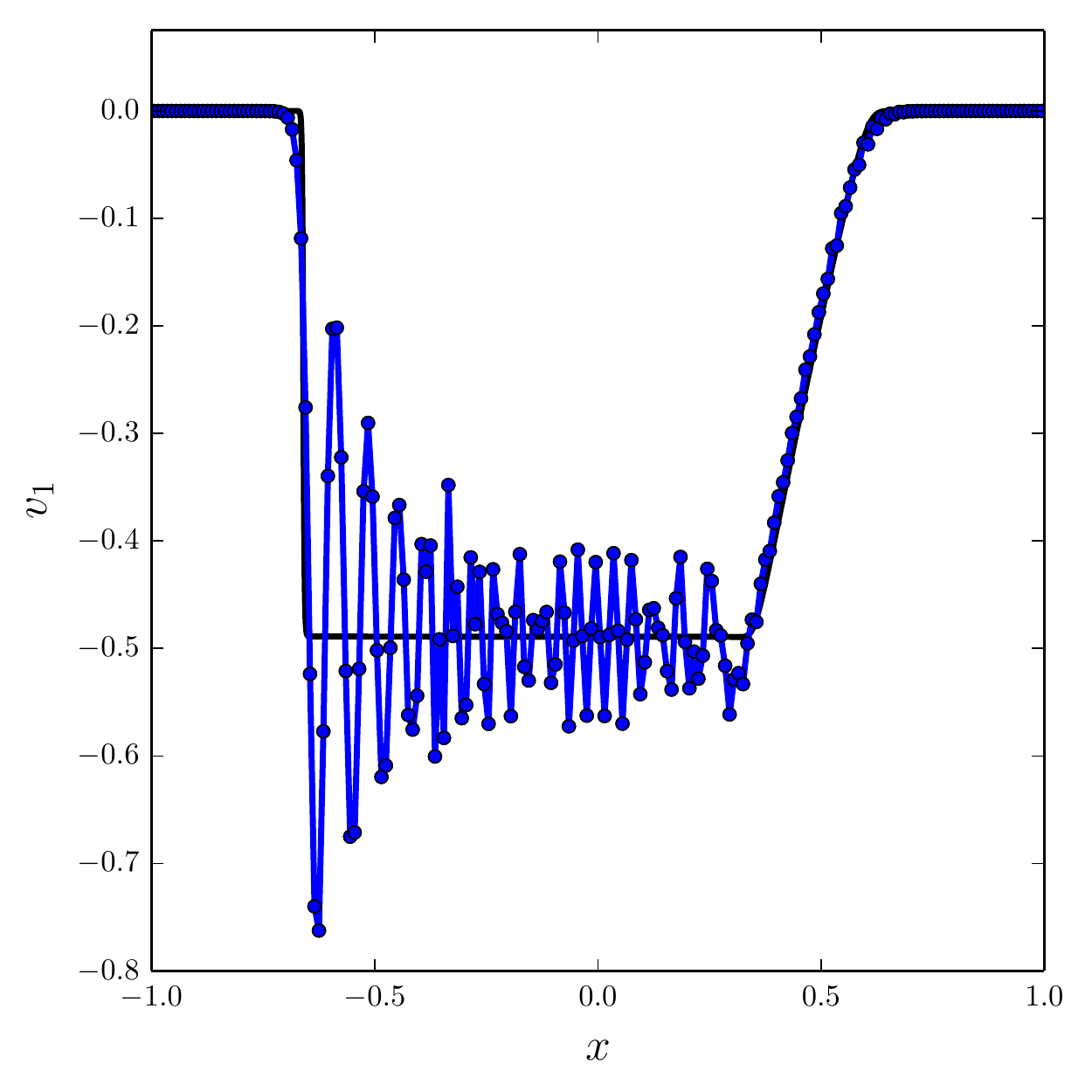}
}
{
\includegraphics[scale=0.575]{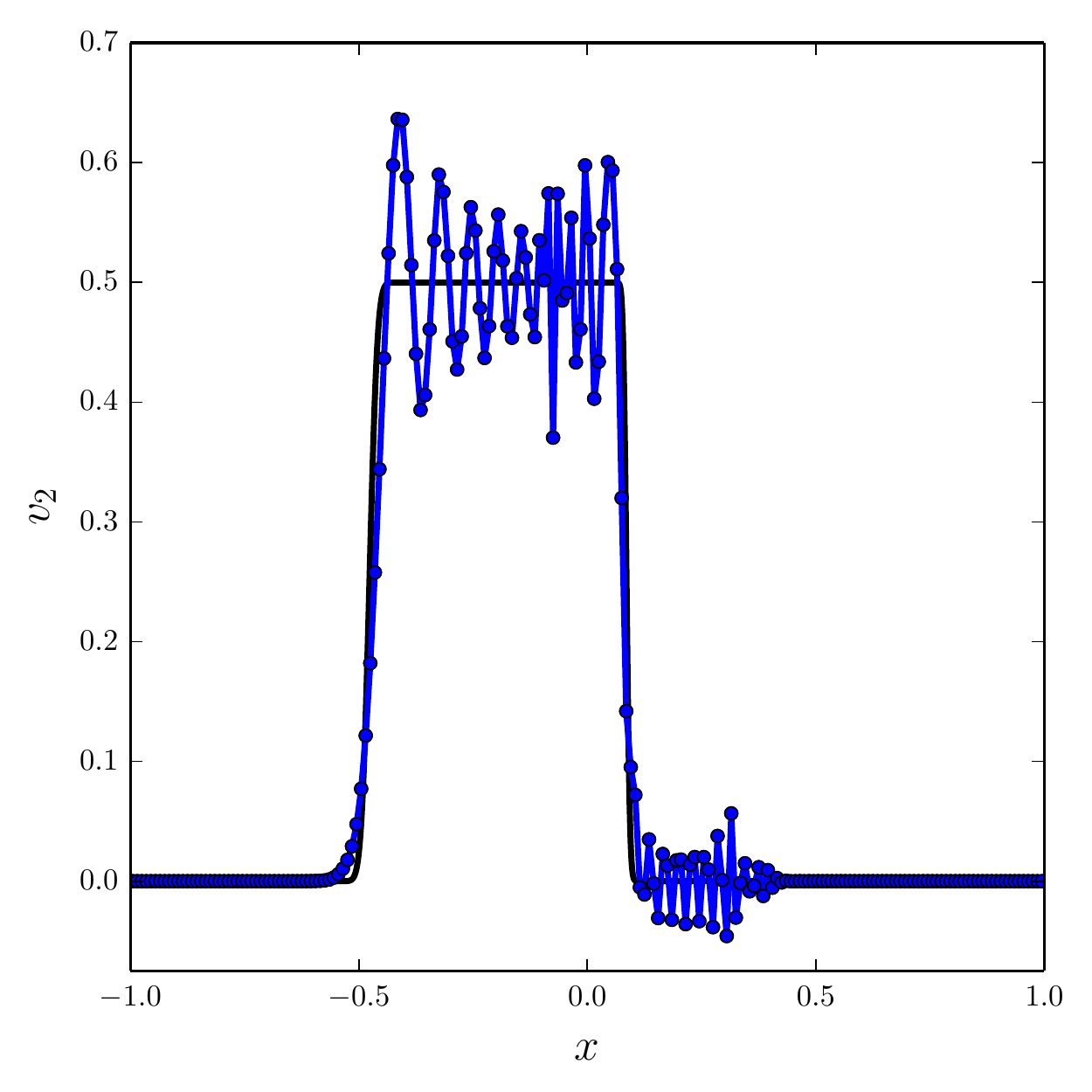}
}
{
\includegraphics[scale=0.575]{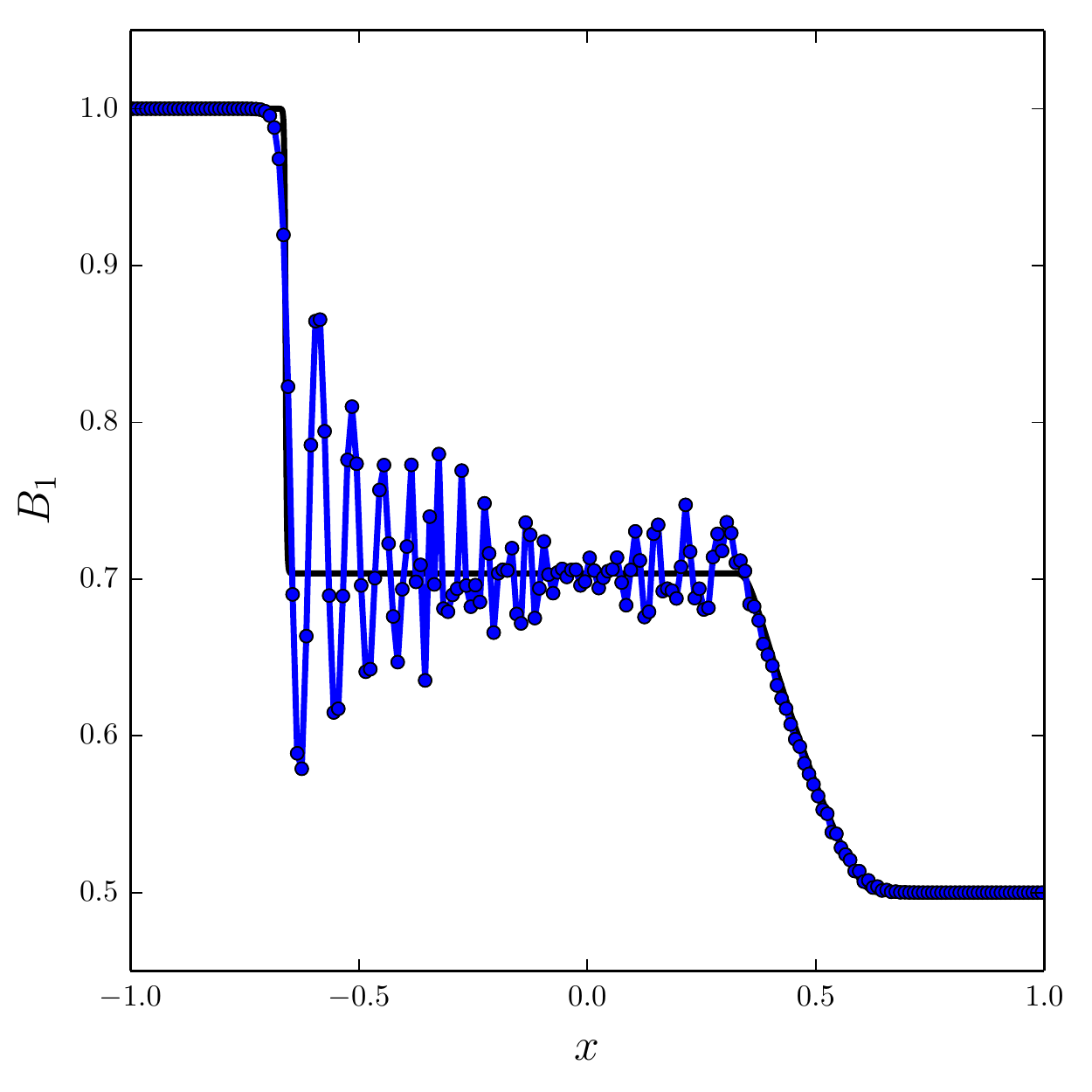}
}
{
\includegraphics[scale=0.575]{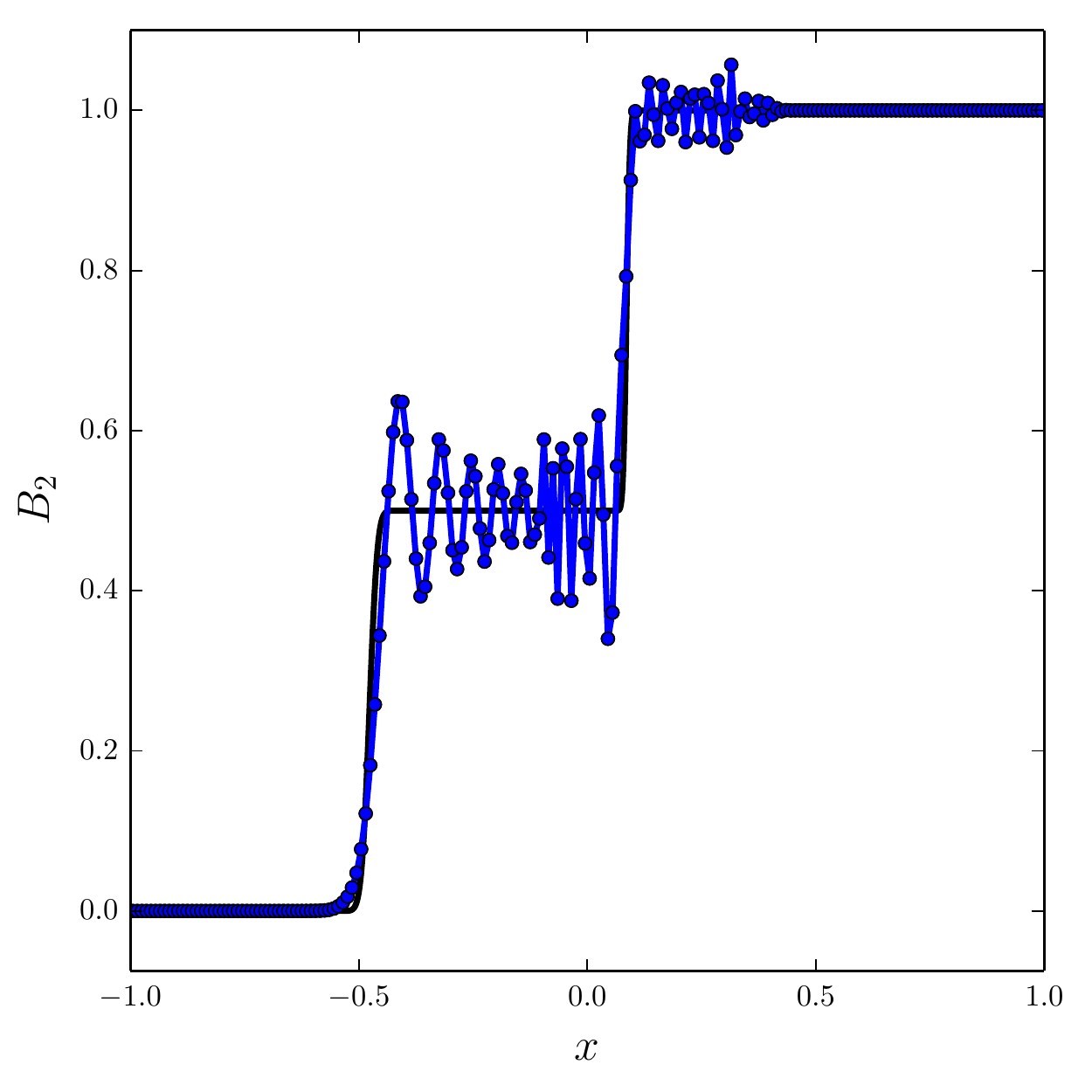}
}
\caption{In blue we present the entropy conserving approximations of the quantities $h$, $v_1$, $v_2$, $B_1$, and $B_2$ at T = 0.4. Solid black represents the reference solution.}
\label{fig:EC}
\end{center}
\end{figure}

\subsection{Entropy Stable Riemann Problem}\label{ESRiemann}

Next, we use the ES1 and ES2 schemes to compute the solution of \eqref{StrongRiemann} with inflow/outflow type boundary conditions. Again, the computation is performed on 100 regular grid cells with CFL number 0.1 up to a final time of $T=0.4$ with a reference solution created using a high-resolution run of the ES1 scheme on 5000 grid cells. We note that the same reference solution could be obtained from a high-resolution computation of the ES2 scheme. We present the solution computed with the entropy stable methods against the reference solution in Fig. \ref{fig:ES}. 

Our computed results for the computed fluid height $h$, velocities $v_1$, $v_2$, and magnetic field $B_1$, $B_2$ shown in Fig. \ref{fig:ES} compare well to the literature \cite{desterck2001,rossmanith2002}. We see that, as was discussed in Sec. \ref{Sec:StableFlux}, the Lax-Friedrichs type stabilization introduces more dissipation into the approximation.
\begin{figure}[!ht]
\begin{center}
{
\includegraphics[scale=0.575]{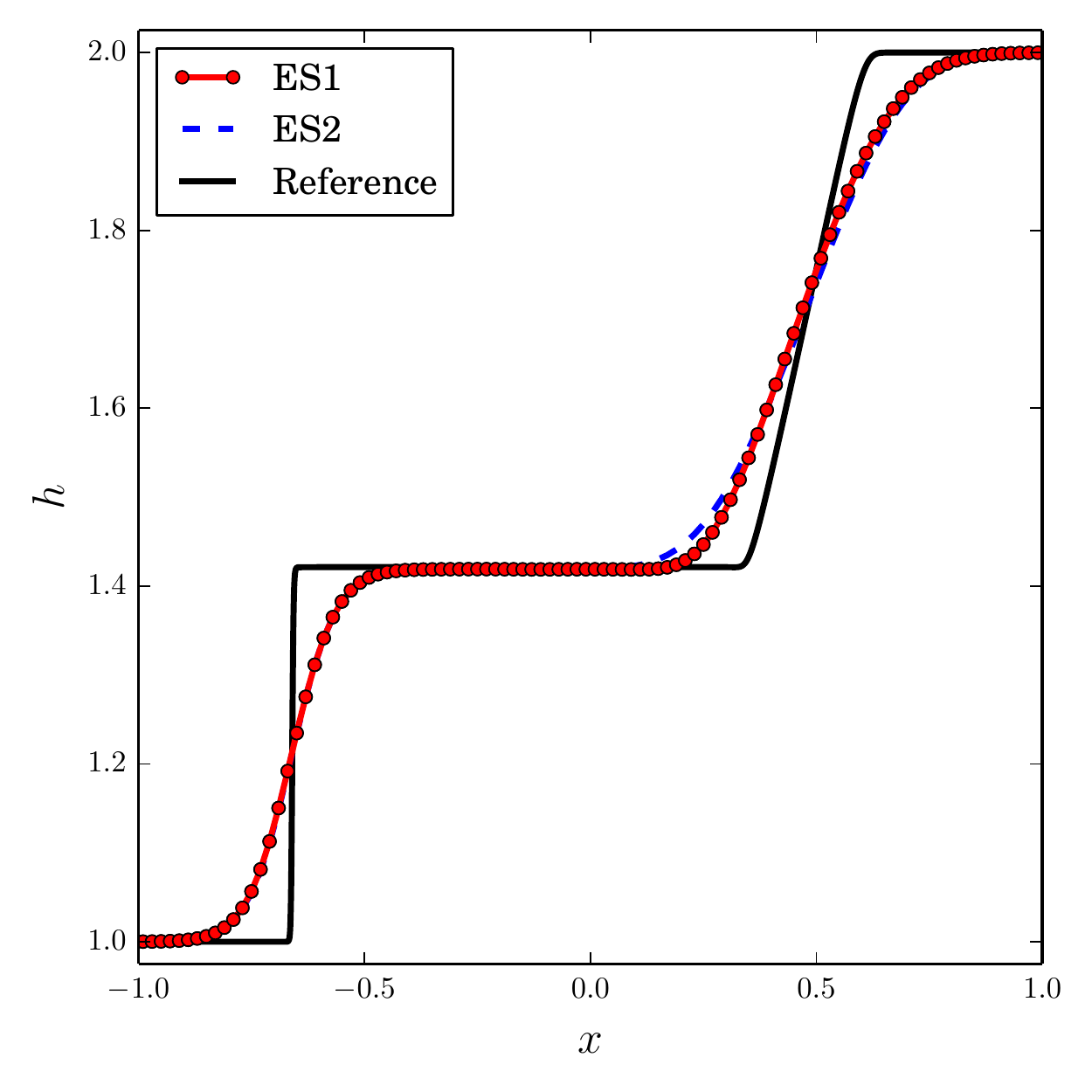}
}
\\
{
\includegraphics[scale=0.575]{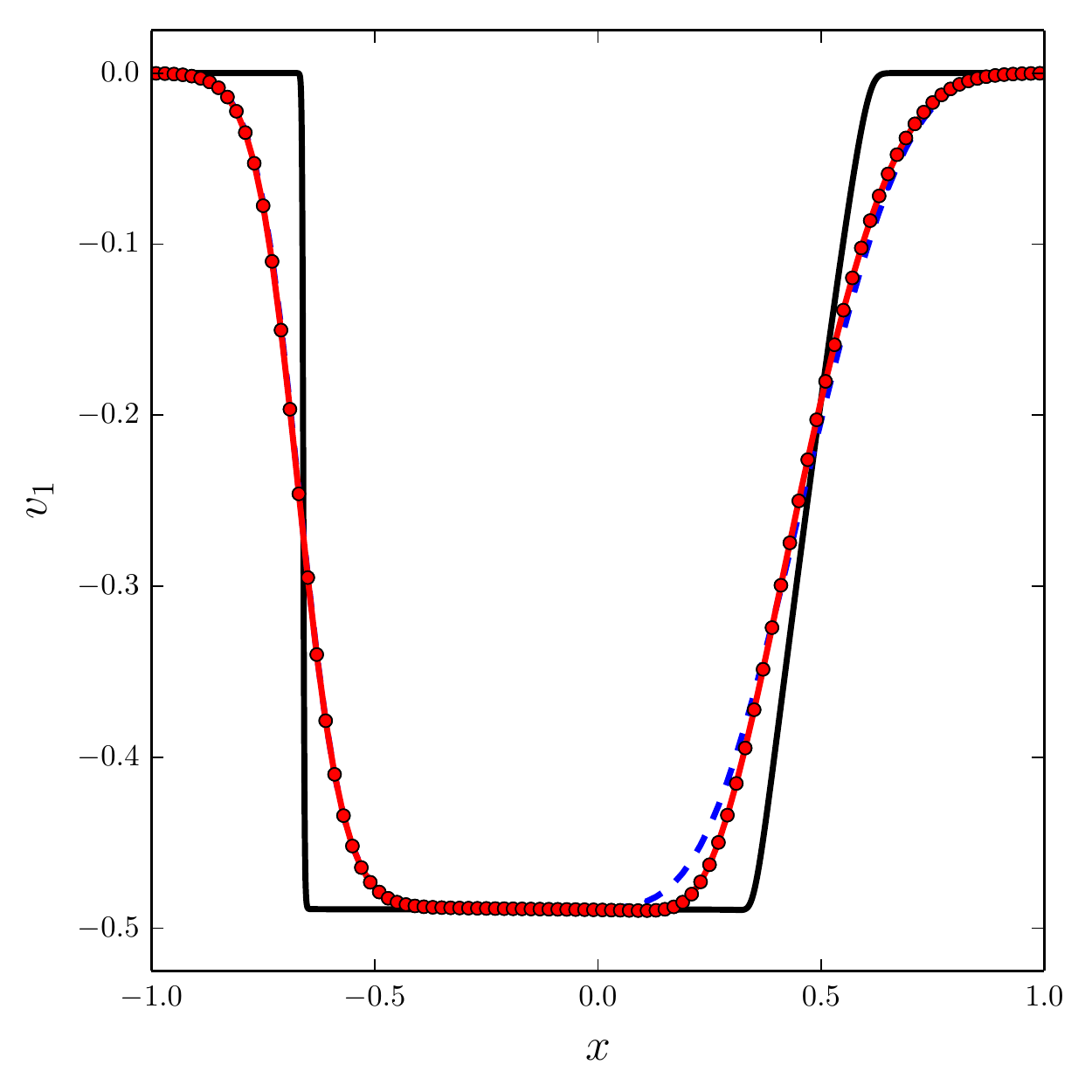}
}
{
\includegraphics[scale=0.575]{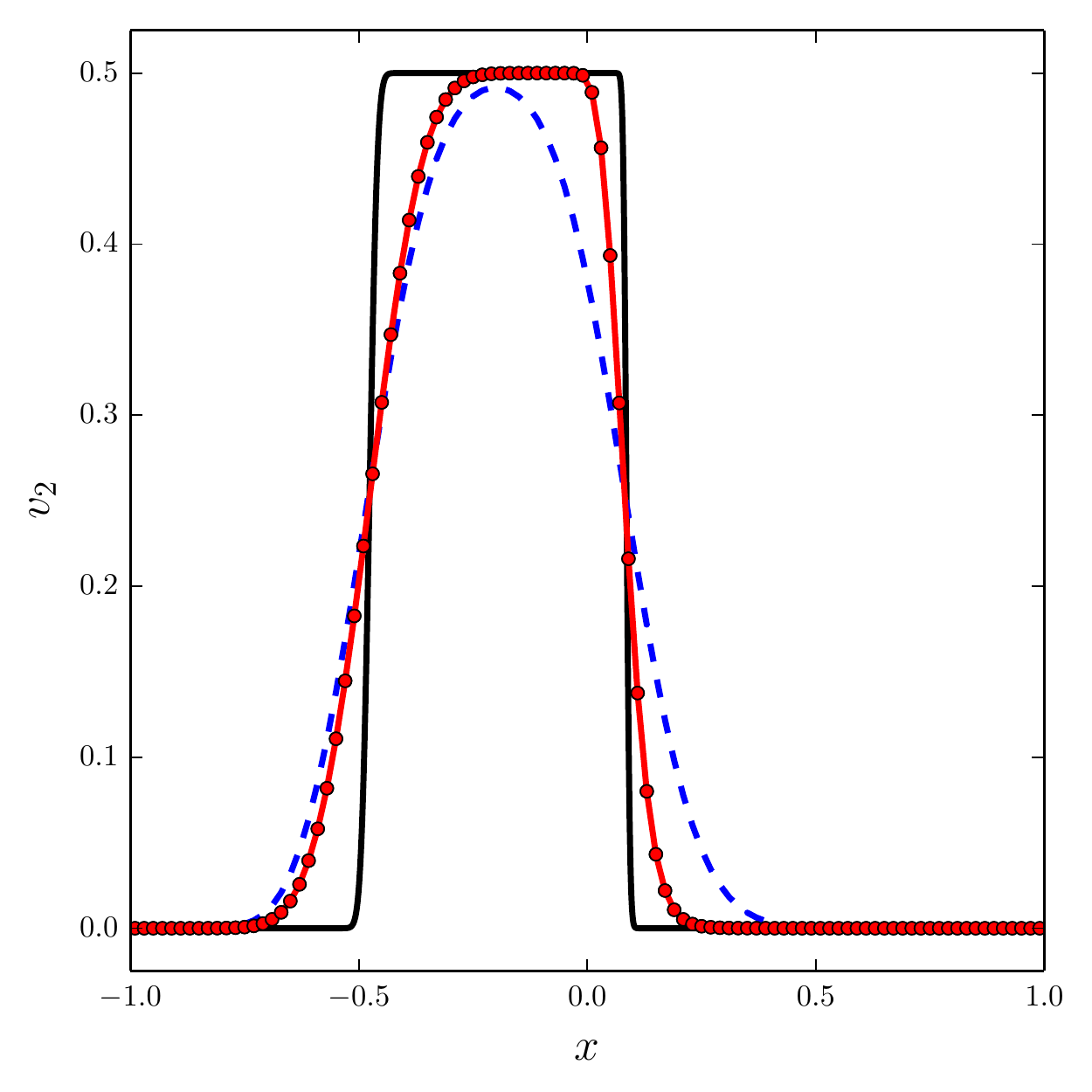}
}
{
\includegraphics[scale=0.575]{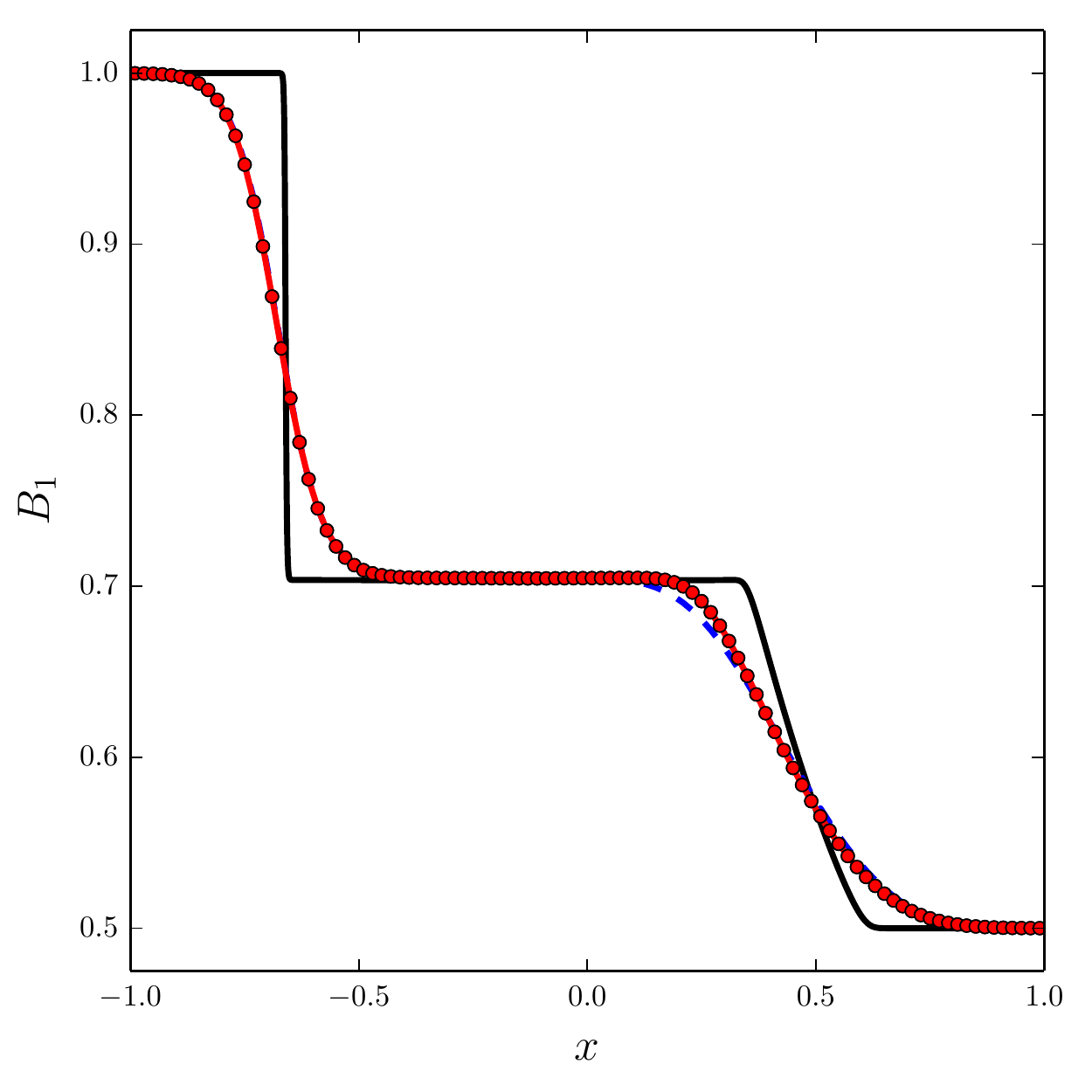}
}
{
\includegraphics[scale=0.575]{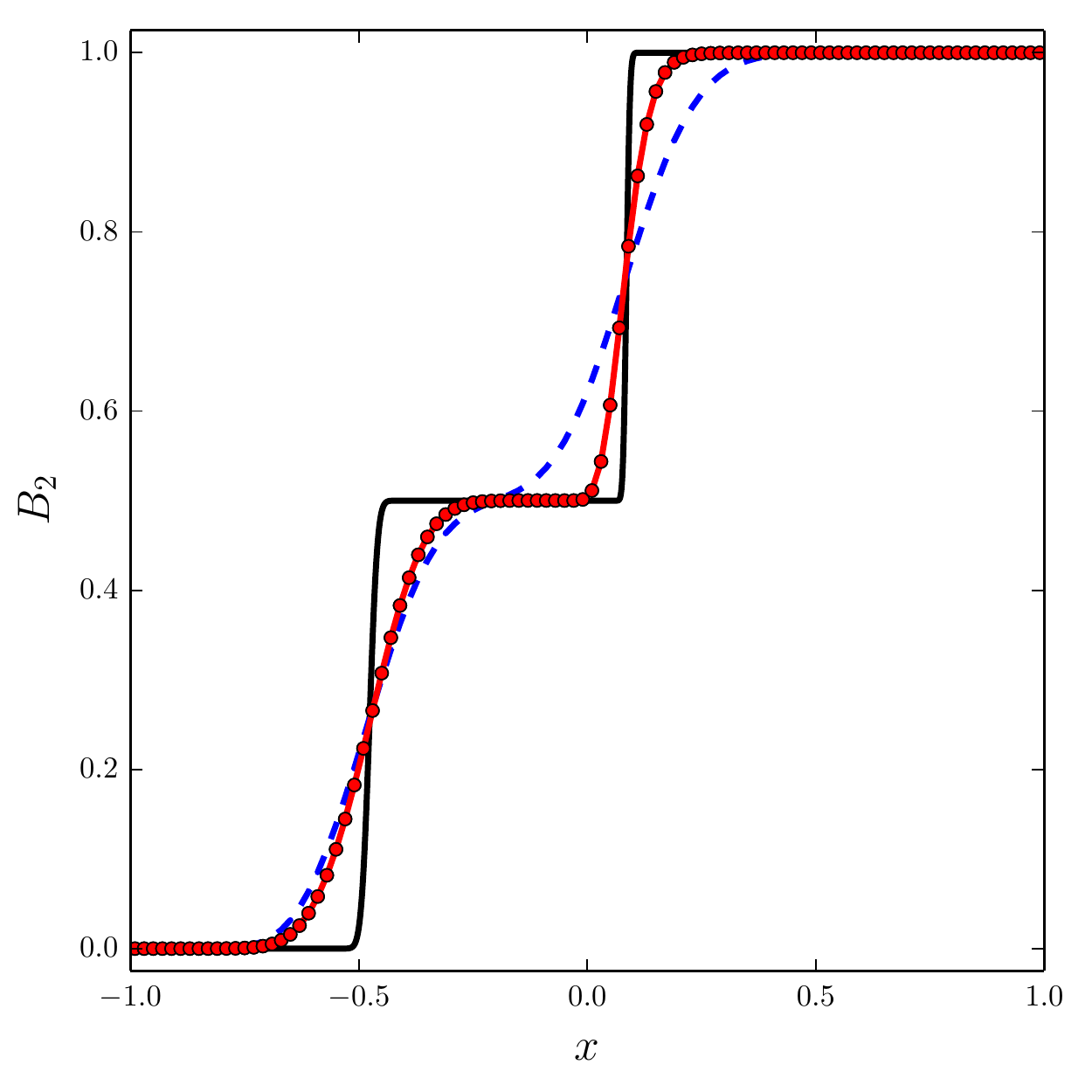}
}
\caption{The entropy stable approximations of the quantities $h$, $v_1$, $v_2$, $B_1$, and $B_2$ at T = 0.4. Solid black is the reference solution, dashed blue is the ES2 scheme, and red with knots is the ES1 scheme.}
\label{fig:ES}
\end{center}
\end{figure}

{\color{black}{As a final test we compare the ES1 scheme to a standard Roe scheme. The development of Roe type schemes for the SWMHD equations has been considered by several authors and complete details can be found in \cite{desterck2001,kemm2014,rossmanith2002}. The comparative computation used 100 regular grid cells with  a CFL number of 0.1 up integrated to $T=0.4$.  A reference solution was created from a high-resolution run of the ES1 scheme on 5000 grid cells. The computed results of the schemes are presented in Fig. \ref{fig:RoeCompare}. We see that for the same number of degrees of freedom the ES1 scheme is less dissipative than the Roe scheme. This is particularly evident in the approximation of quantities in the tangential direction, $v_2$ and $B_2$. We also note that the Roe scheme is more dissipative than the ES2 Lax-Friedrichs type entropy stable approximation which can be seen from a comparison of the results in Fig. \ref{fig:ES} and Fig. \ref{fig:RoeCompare}.}}
\begin{figure}[!ht]
\begin{center}
{
\includegraphics[scale=0.575]{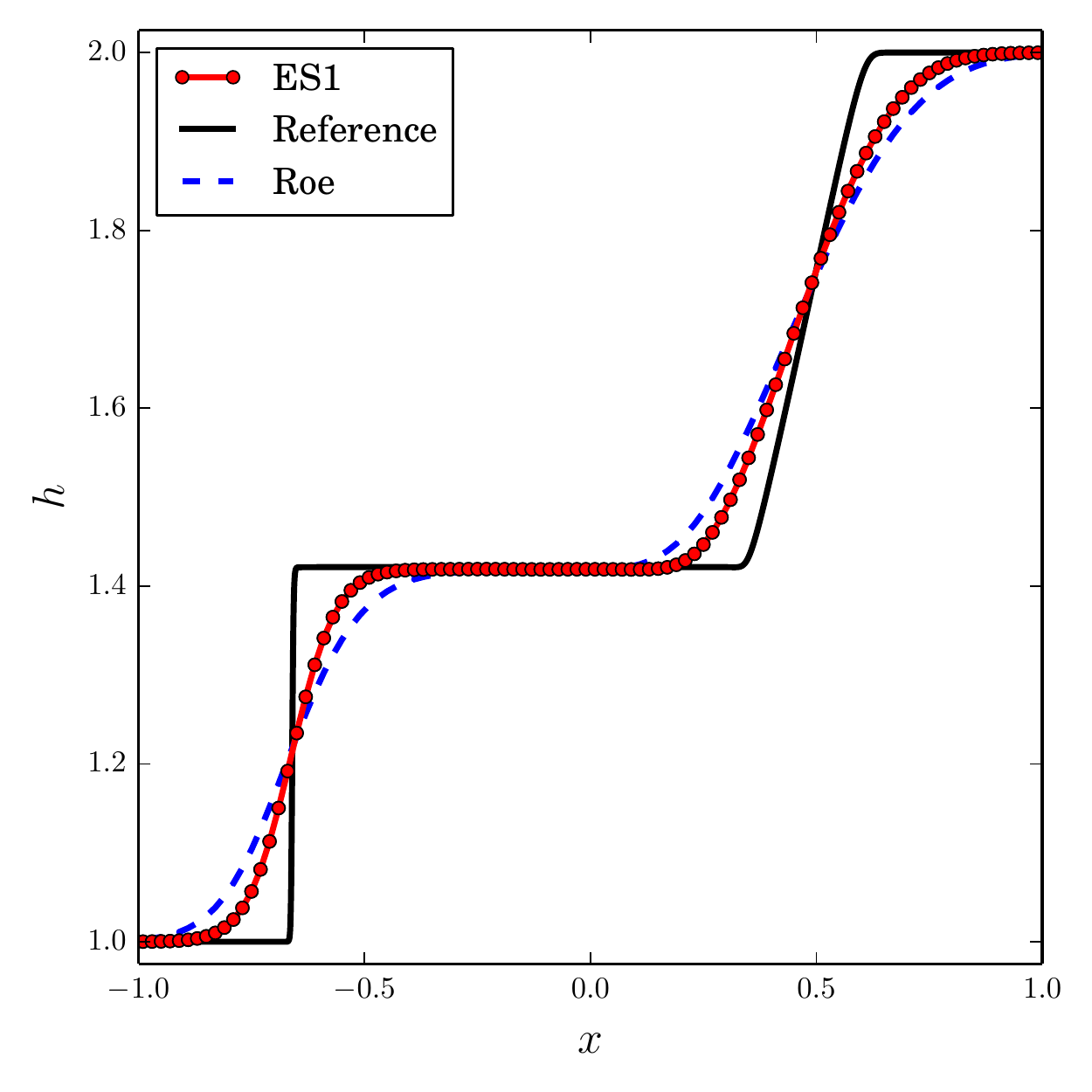}
}
\\
{
\includegraphics[scale=0.575]{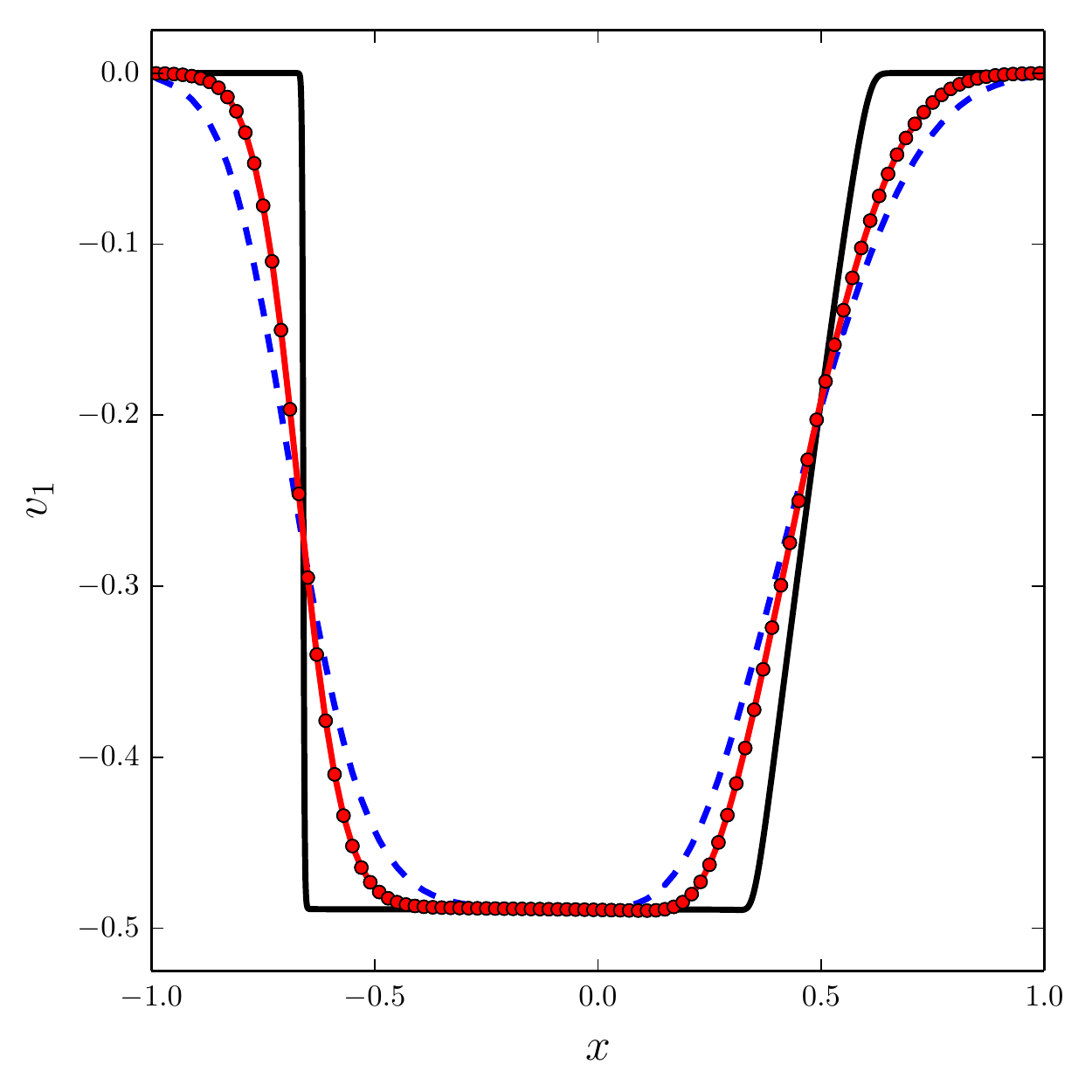}
}
{
\includegraphics[scale=0.575]{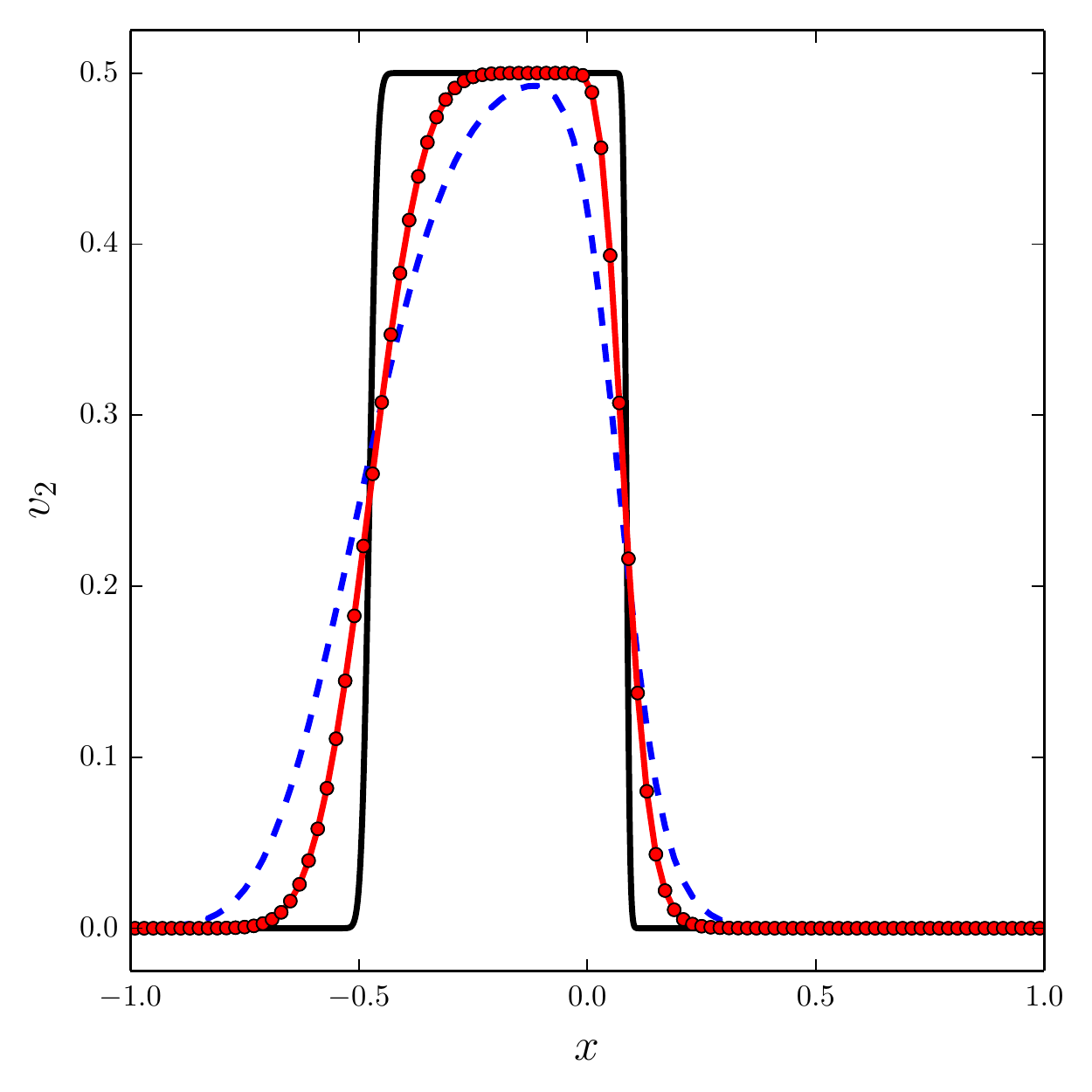}
}
{
\includegraphics[scale=0.575]{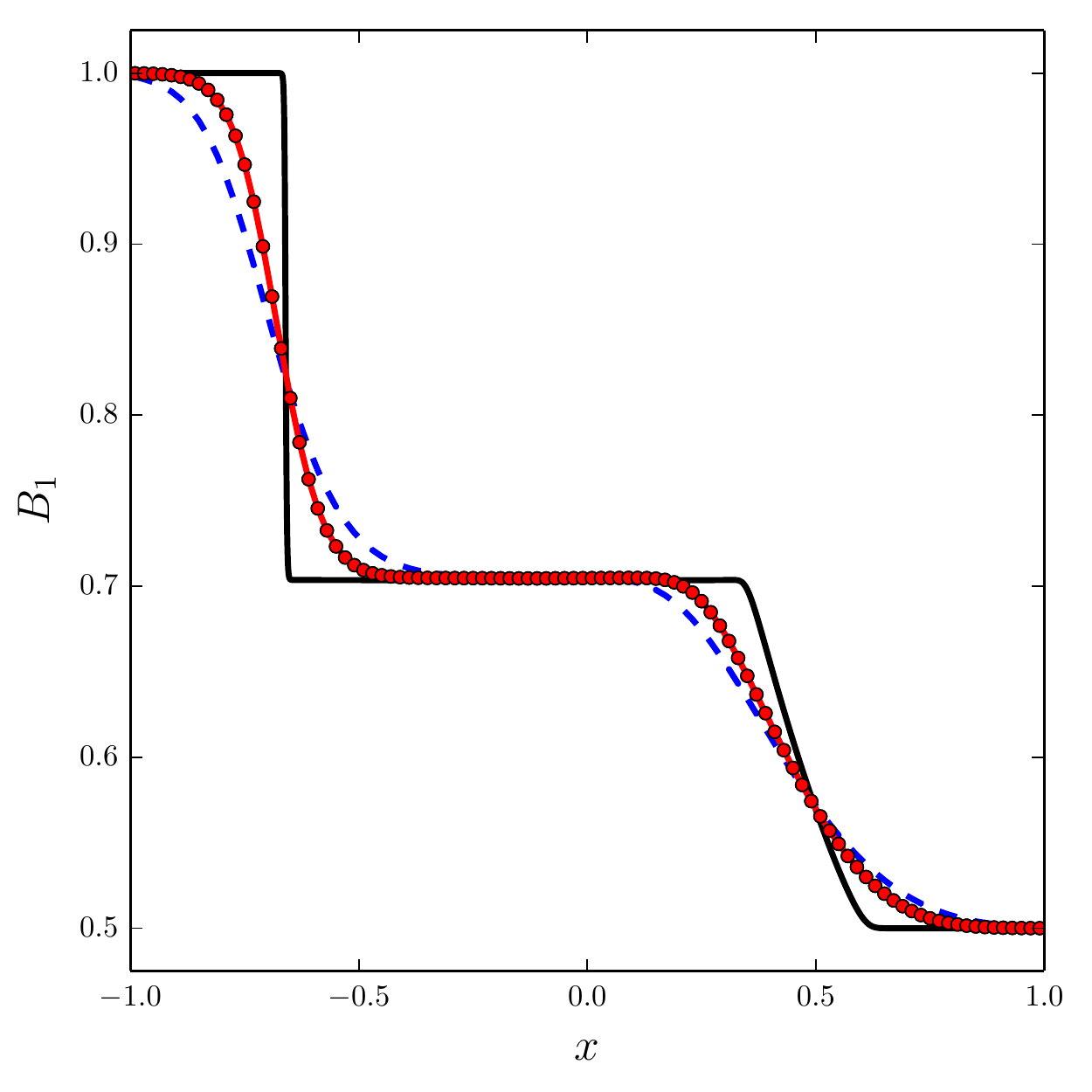}
}
{
\includegraphics[scale=0.575]{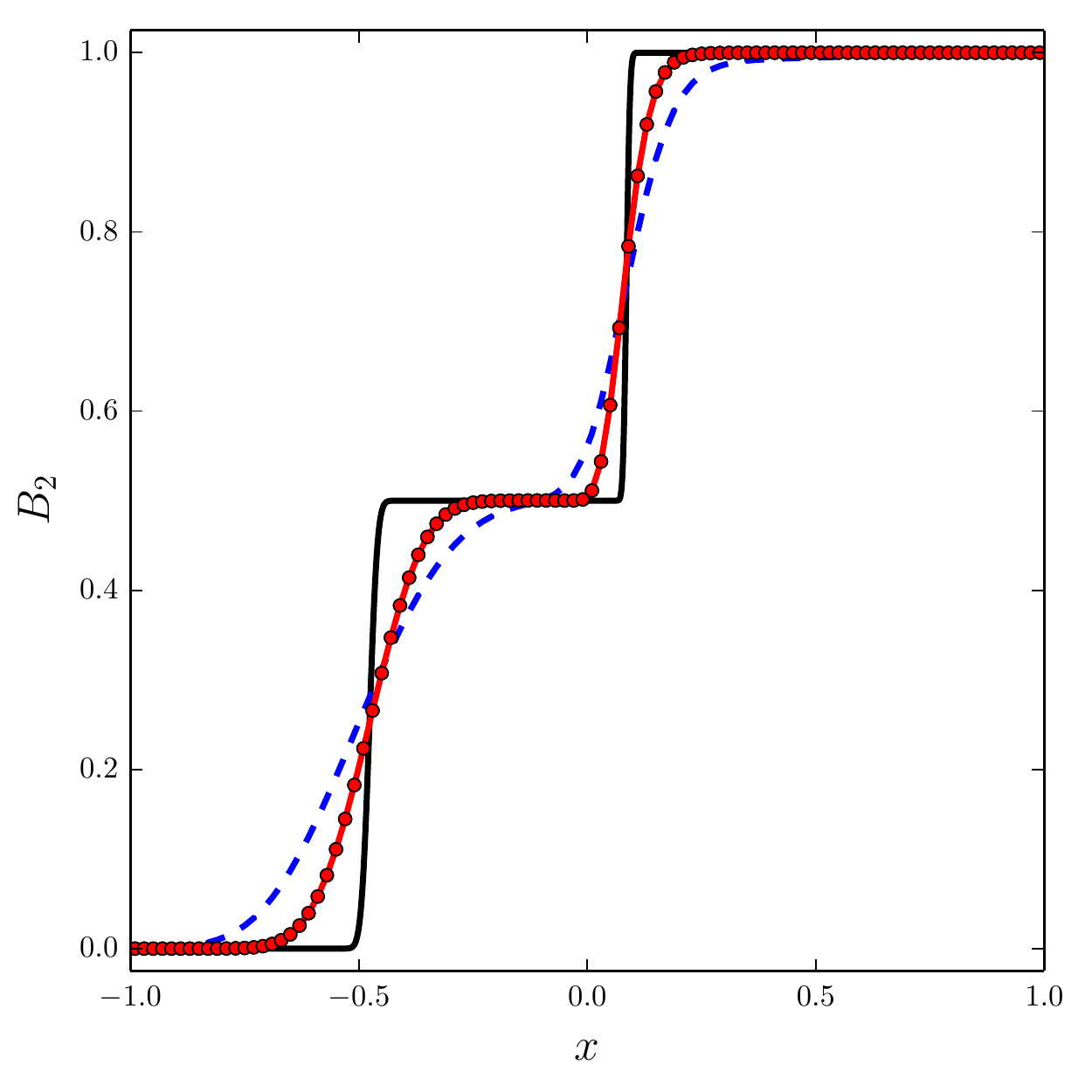}
}
\caption{The entropy stable results compared against a Roe scheme for the quantities $h$, $v_1$, $v_2$, $B_1$, and $B_2$ at T = 0.4. Solid black is the reference solution, dashed blue is the Roe scheme, and red with knots is the ES1 scheme.}
\label{fig:RoeCompare}
\end{center}
\end{figure}

\section{Concluding Remarks}\label{conclusion}

In this work we present a novel, affordable, entropy stable flux for the shallow water MHD equations. The derivations required us to relax the involution condition such that $\partial_x(hB_1)\approx 0$ within finite precision error. Under this less restrictive assumption we were able to derive an entropy conserving flux, denoted $\vec{f}^{*,ec}$. Special care had to be taken for the discretization of the source term in order to guarantee discrete entropy conservation. Because entropy conserving approximations can suffer breakdown at shocks we extended our analysis and derived two stabilizing dissipation terms that we apply to the entropy conserving flux. We finally used a variety of numerical test examples to demonstrate and underline the theoretical findings.

{\color{black}{The derivation of the entropy conserving and entropy stable numerical fluxes in this paper focused on the one-dimensional SWMHD equations. The restriction to one spatial dimension was because the analysis proved to be quite involved. However, the discussion was general, so the derivations in this paper readily extend to provably entropy conserving and entropy stable approximations of the SWMHD equations on multi-dimensional Cartesian grids. In the Appendix we provide details on the multidimensional flux functions.

It is well known that the issue of divergence-free constraint and the accuracy of an approximation is more serious in two dimensions. As was discussed in Sec. \ref{GoverningEqns} the Janhunen source term acts, analogously, to a hyperbolic divergence cleaning method where the error in the involution term is radiated out of the computational domain. To verify the utility of our method we consider a two dimensional example similar to the ``rotor problem'' of Balsara and Spicer \cite{balsara1999} and considered by T\'{o}th \cite{toth2000} for the ideal MHD equations. The computational domain is $\Omega = [-1,1]\times[-1,1]$ with initial data
\begin{equation}\label{Rotor}
\begin{bmatrix}h\\v_1\\v_2\\B_1\\B_2\end{bmatrix} = \left\{
\begin{array}{lc}
\left[10,-y,x,0.1,0\right]^T, & \textrm{if}\quad \|\vec{x}\|_2 < 0.1, \\ 
\left[1,0,0,1,0\right]^T, & \textrm{if}\quad \|\vec{x}\|_2 > 0.1,
\end{array}
\right.
\end{equation}
integrated up to the final time $T=0.2$. The computational results of the two dimensional ES1 scheme with 200 regular grid cells in each direction are presented in Fig. \ref{fig:2D}. We note that the computed fluid height $h$ and velocity components $v_1$ and $v_2$ compare well with previous results found in the literature \cite{kroeger2005}. As we expect, because we use a different divergence cleaning procedure, our results for the $B_1$ and $B_2$ magnetic fields differ slightly from those computed in \cite{kroeger2005}. 

In this paper we have demonstrated that the entropy stable approximations for the SWMHD equations are competitive with existing solvers. We also showed that the methods readily extend to multiple dimensions because the source term we need for entropy conservation also acts as a divergence cleaning technique.}}
\begin{figure}[!ht]
\begin{center}
{
\includegraphics[scale=0.56]{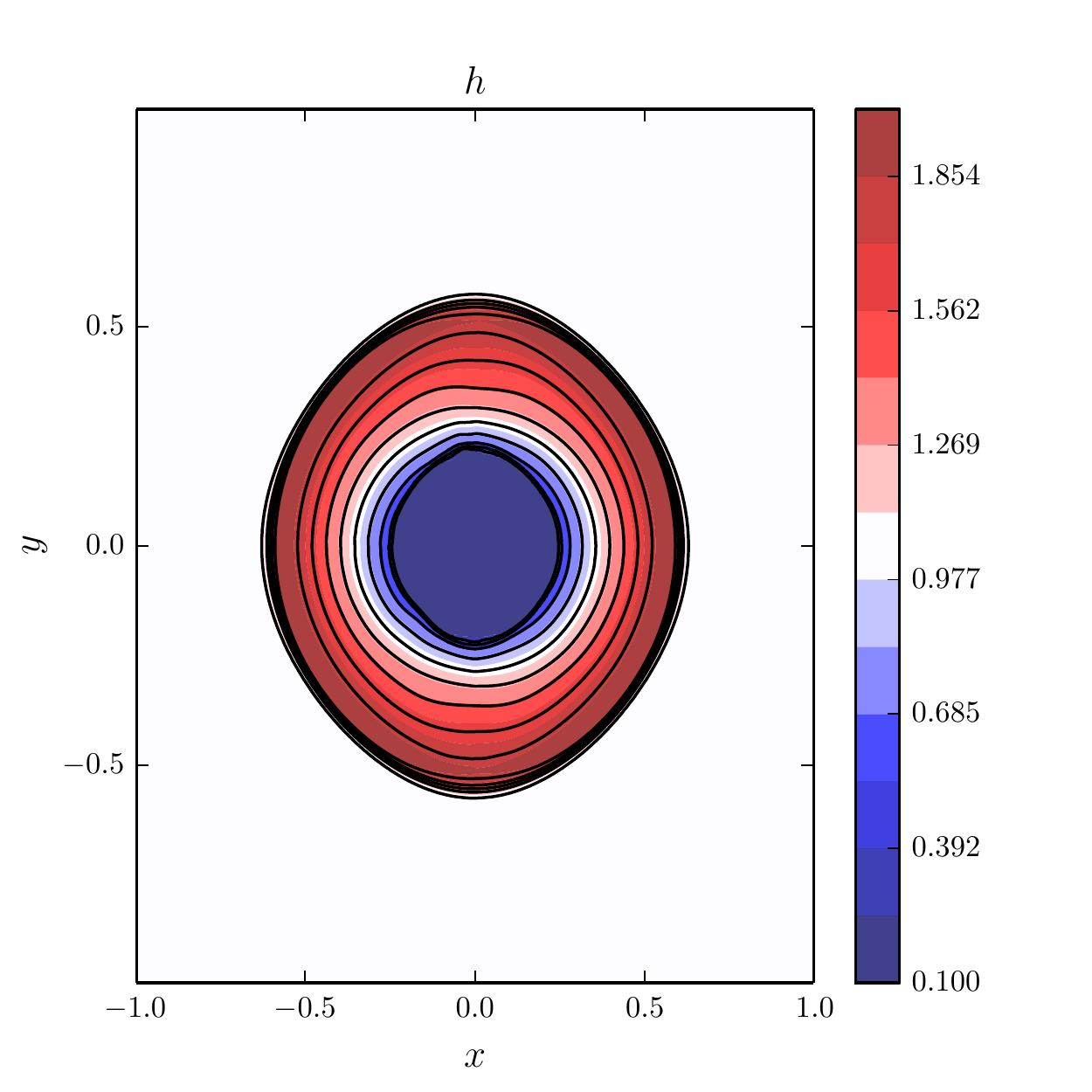}
}
\\
{
\includegraphics[scale=0.56]{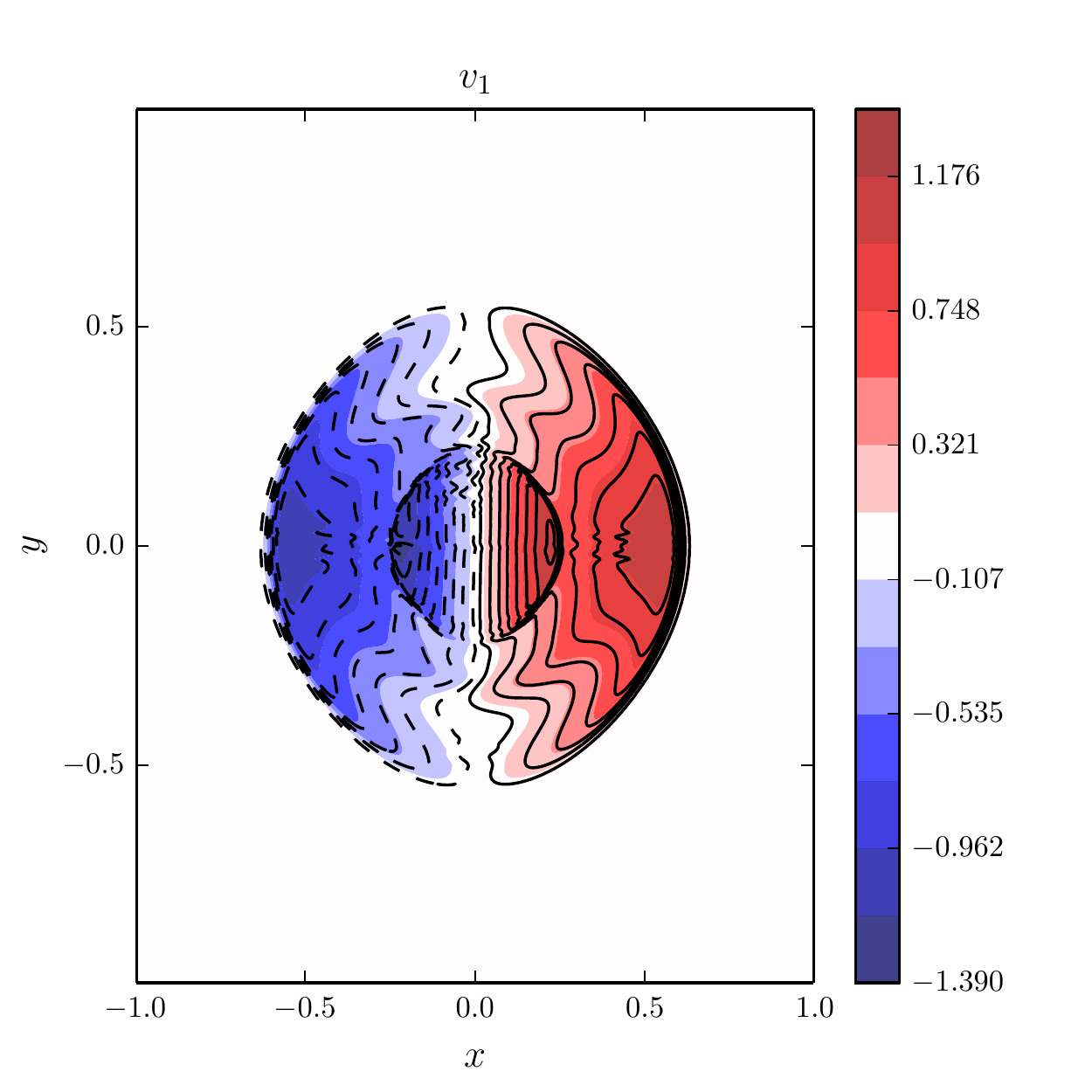}
}
{
\includegraphics[scale=0.56]{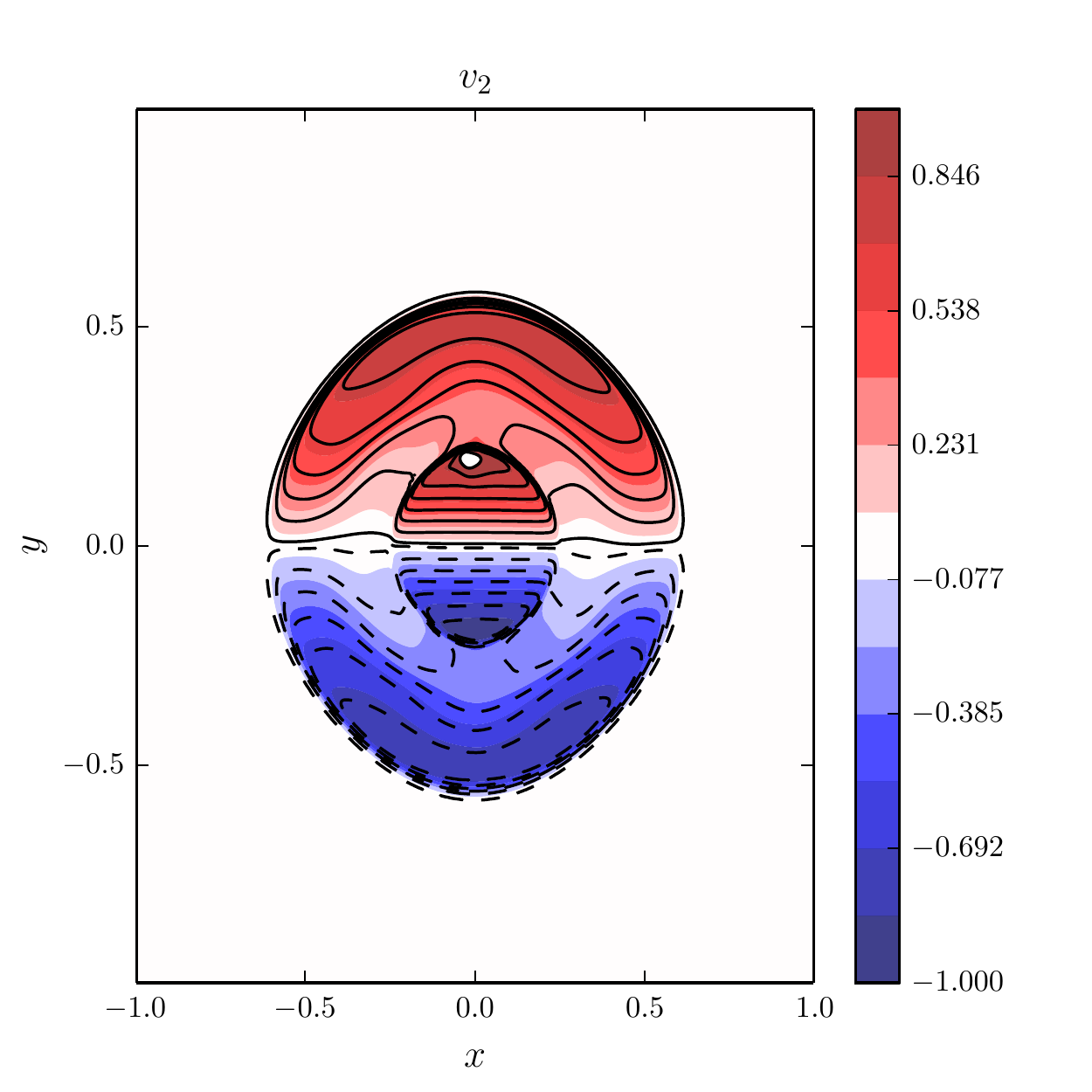}
}
{
\includegraphics[scale=0.56]{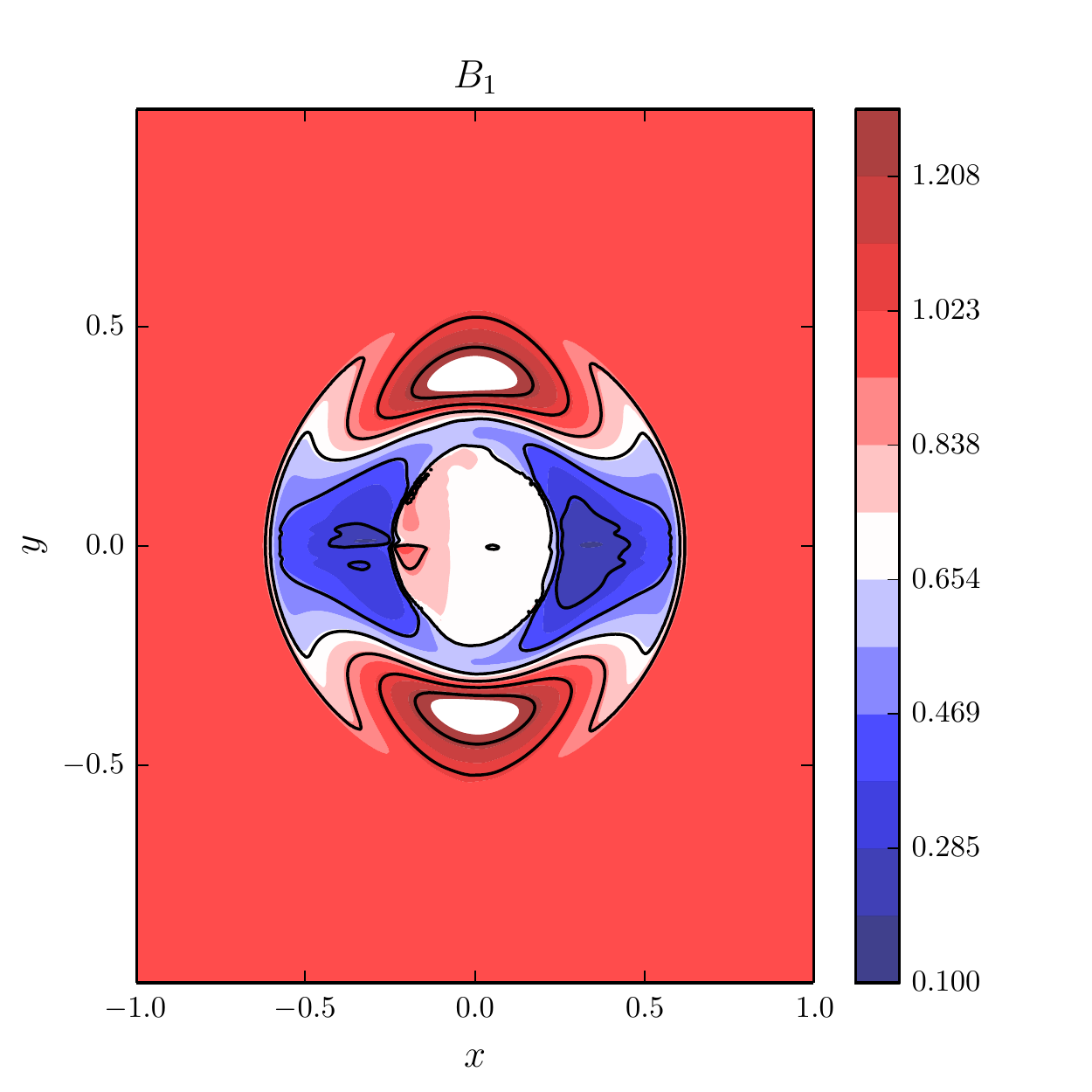}
}
{
\includegraphics[scale=0.56]{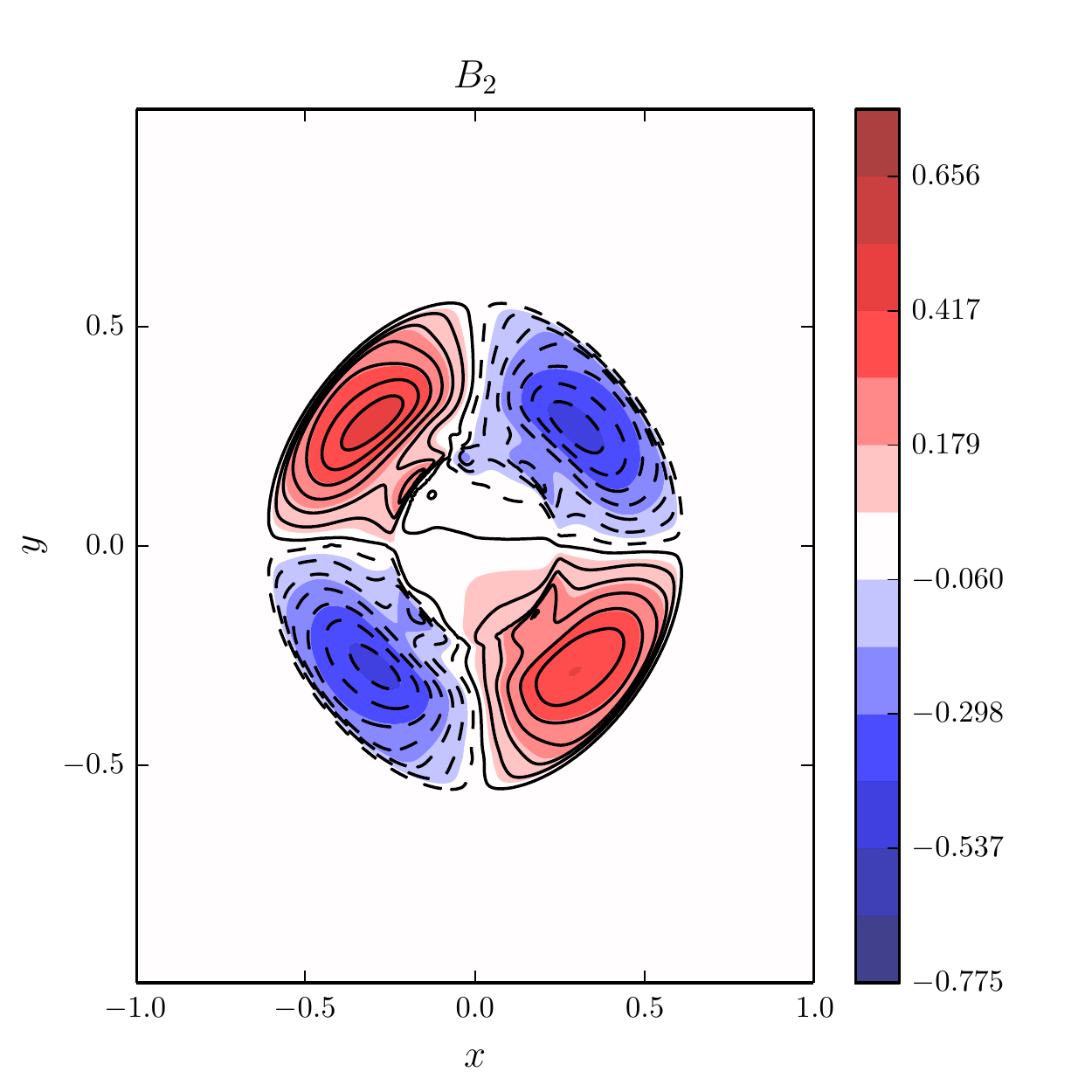}
}
\caption{The computed solution to the ``rotor problem'' \eqref{Rotor} using the ES1 scheme for the quantities $h$, $v_1$, $v_2$, $B_1$, and $B_2$ at T = 0.2. Dashed lines are negative contour lines.}
\label{fig:2D}
\end{center}
\end{figure}

\appendix
{\color{black}{\section{Flux Functions in Two Spatial Dimensions}\label{2DFluxes}

We note that, for the discussion of the two dimension fluxes we suppress the $i$ index on the multidimensional approximation for the purpose of clarity. 

\subsection{Entropy Conservative Flux}

To develop the entropy conservative flux in the $y-$direction we first note that the entropy potential
\begin{equation}\label{EntropyPotentialY}
\psi_y = \vec{v}\cdot\vec{g} - G = \frac{g}{2}h^2v_2 - hB_2(v_1B_1 + v_2B_2),
\end{equation}
where we have the entropy flux in the $y-$direction
\begin{equation}
G = gh^2v_2 + \frac{1}{2}\left(hv_1^2v_2 + hv_2^3 + hv_2B_1^2 - hv_2B_2^2\right) - hv_1B_1B_2.
\end{equation}
Note that the discrete entropy conservation condition \eqref{entropyConservationCondition1} has the same structure in each Cartesian direction. Lastly, the source term contributes symmetrically to each direction. With a proof analogous of that for $\vec{f}^{*,ec}$ in Sec. \ref{EntropyFlux} we present the entropy conserving numerical flux for the $y-$direction.
\begin{cor}(Entropy Conserving Numerical Flux: $y-$direction)
If we discretize the source term in the finite volume method to contribute to each element as 
\begin{equation}
\vec{s}_j = \frac{1}{2}\left(\vec{s}_{j+\tfrac{1}{2}} + \vec{s}_{j-\tfrac{1}{2}}\right) = -\frac{1}{2}\left(
\jump{hB_2}_{j+\tfrac{1}{2}}
\begin{bmatrix}
0\\
0\\
0\\
\frac{\average{v_1B_1}}{\average{\Delta x B_1}}\\[0.1cm]
\frac{\average{v_2B_2}}{\average{\Delta x B_2}}
\end{bmatrix}_{j+\tfrac{1}{2}}
+
\jump{hB_2}_{j-\tfrac{1}{2}}
\begin{bmatrix}
0\\
0\\
0\\
\frac{\average{v_1B_1}}{\average{\Delta x B_1}}\\[0.1cm]
\frac{\average{v_2B_2}}{\average{\Delta x B_2}}
\end{bmatrix}_{j-\tfrac{1}{2}}
\right),
\end{equation}
then we can determine a discrete, entropy conservative flux of the form
\begin{equation}
\vec{g}^{*,ec} = \begin{bmatrix}
\average{h}\average{v_2} \\[0.1cm]
\average{h}\average{v_1}\average{v_2} - \average{hB_2}\average{B_1}\\[0.1cm]
\average{h}\average{v_2}^2 + \frac{1}{2}g\left\{\!\!\!\left\{h^2\right\}\!\!\!\right\} - \average{hB_2}\average{B_2} \\[0.1cm]
\average{h}\average{v_2}\average{B_1}-\average{hB_2}\average{v_1} \\[0.1cm]
\average{h}\average{v_2}\average{B_2}-\average{hB_2}\average{v_2}
\end{bmatrix}.
\end{equation}
\end{cor}
\subsection{Entropy Stable Fluxes}
Just as was done in Sec. \ref{Sec:StableFlux} we can create 2D entropy stable flux functions. We, again, use the flux Jacobian altered by the Powell source term to create the dissipation term:
\begin{equation}\label{alteredFluxJacobianY}
\doublehat{\matrix{B}} = \begin{bmatrix} 
0 & 0 & 1 & 0 & 0  \\ 
-v_1v_2+B_1B_2 & v_2 & v_1 & -B_2 & 0 \\
gh - v_2^2 + B_2^2 & 0 & 2v_2 & 0 & -B_2 \\
v_1B_2 - v_2B_1 & -B_2 & B_1 & v_2 & 0 \\
0 & 0 & 0 & 0 & v_2 \\
\end{bmatrix},
\end{equation}
equipped with a full set of eigenvalues
\begin{equation}\label{eigenvaluesY}
\lambda_1 = v_2-c_g,\quad\lambda_2 = v_2-B_2,\quad\lambda_3 = v_2,\quad\lambda_4 = v_2+B_2,\quad\lambda_5 = v_2+c_g,
\end{equation}
and right eigenvectors
\begin{equation}\label{rightEVY}
\doublehat{\matrix{R}}_y = \begin{bmatrix}
1 & 0 & 1 & 0 & 1 \\[0.1cm]
v_1 & 1 & v_1 & 1 & v_1\\[0.1cm]
v_2-c_g & 0 & v_2 & 0 & v_2+c_g \\[0.1cm]
B_1 & 1 & B_1 & -1 & B_1 \\[0.1cm]
0    & 0 & \frac{c_g^2}{B_2} & 0 & 0 
\end{bmatrix},
\quad
\doublehat{\matrix{L}}_y = \doublehat{\matrix{R}}_y^{-1},
\end{equation}
where $c_g^2 = gh+B_2^2$ is the magnetogravity wave speed.
\begin{cor} (Entropy Stable 1 (ES1): $y-$direction) If we apply the diagonal scaling matrix
\begin{equation}\label{scalingMatrixY}
\matrix{T}_y = diag\left(\frac{c}{c_g\sqrt{2g}}\,,\,\frac{c}{\sqrt{2g}}\,,\,\frac{B_2}{c_g\sqrt{g}}\,,\,\frac{c}{\sqrt{2g}}\,,\,\frac{c}{c_g\sqrt{2g}}\right),
\end{equation}
to the matrix of right eigenvectors $\doublehat{\matrix{R}}_y$ \eqref{rightEVY}, then we obtain the Merriam identity \cite{merriam1989} (Eq. 7.3.1 pg. 77) 
\begin{equation}\label{MerriamIdentityY}
\matrix{H} = \widetilde{\matrix{R}}_y\widetilde{\matrix{R}}_y^T = \left(\doublehat{\matrix{R}}_y\matrix{T}_y\right) \left(\doublehat{\matrix{R}}_y\matrix{T}_y\right)^T = \doublehat{\matrix{R}}_y\matrix{S}_y\doublehat{\matrix{R}}_y^T,
\end{equation}
that relates the right eigenvectors of $\doublehat{\matrix{B}}$ to the entropy Jacobian matrix \eqref{entropyJacobian}. For convenience, we introduce the diagonal scaling matrix $\matrix{S}_y=\matrix{T}\,^2\!\!\!\!_y$ in \eqref{MerriamIdentityY}. We then have the guaranteed entropy stable flux interface contribution
\begin{equation}\label{minimalDissY}
\vec{g}^{*,ES1} = \vec{g}^{*,ec}-\frac{1}{2}\matrix{D}\jump{\vec{u}}=\vec{g}^{*,ec} - \frac{1}{2} \doublehat{\matrix{R}}_y|\doublehat{\boldsymbol\Lambda}|\matrix{S}_y\doublehat{\matrix{R}}_y^T\jump{\vec{q}}.
\end{equation}

\end{cor}

\begin{rem} \textit{(Entropy Stable 2 (ES2): $y-$direction)}
If we choose the dissipation matrix to be
\begin{equation}
\matrix{D} = |\lambda_{max}|\matrix{I},
\end{equation}
where $\lambda_{max}$ is the largest eigenvalue of the system from $\widehat{\matrix{B}}$ and $\matrix{I}$ is the identity matrix, then we obtain a local Lax-Friedrichs type interface stabilization
\begin{equation}\label{LFDissY}
\begin{aligned}
\vec{g}^{*,ES2} &= \vec{g}^{*,ec} - \frac{1}{2}|\lambda_{max}|\matrix{H}\jump{\vec{q}}.
\end{aligned}
\end{equation} 
\end{rem}
}}
\bibliographystyle{siam} 
\bibliography{References.bib}

\end{document}